\documentclass[a4paper,10pt,papersize]{article}

\usepackage{amssymb,amsfonts,amsmath,amsthm}
\usepackage[dvips]{graphicx,hyperref}
\usepackage{cite,caption}
\usepackage{mathtools}
\usepackage{booktabs}
\usepackage{mathptmx}
\newcommand{\fz}{\frac}
\newcommand{\prz}[2]{ \frac{\partial{#1}}{\partial{#2}} }
\newcommand{\pz}{\partial}
\newcommand{\lA}{\langle}
\newcommand{\rA}{\rangle}

\newcommand{\ol}{\overline}

\renewcommand{\Omega}{\varOmega}
\renewcommand{\Gamma}{\varGamma}
\renewcommand{\Psi}{\varPsi}
\renewcommand{\Pi}{\varPi}
\newtheorem{Hyp}{Hypothesis}[section]
\newtheorem{Thm}[Hyp]{Theorem}

\newtheorem{Lem}[Hyp]{Lemma}
\newtheorem{Cor}[Hyp]{Corollary}
\newtheorem{Rmk}[Hyp]{Remark}
\newtheorem{Ex}{Example}[section]
\newtheorem{Alg}{Algorithm}[section]

\newcommand{\R}{\mathbb{R}}
\newcommand{\N}{\mathbb{N}}
\newcommand{\sym}{{\rm sym}}

\newcommand{\defeq}{\vcentcolon=}

\newcommand{\argmin}{\mathop{\rm arg\,min}\limits}
\providecommand{\keywords}[1]{\textit{Keywords:} #1}

\allowdisplaybreaks[4]
\title{
The gradient flow structure of an extended Maxwell viscoelastic model and a structure-preserving finite element scheme
}
\author{
Masato~Kimura$^{1}$,
Hirofumi~Notsu$^{1,2,}$\footnote{Corresponding author}\ ,
Yoshimi~Tanaka$^{3}$
and
Hiroki~Yamamoto$^{4}$
}
\date{
\small
$^1$Faculty of Mathematics and Physics, Kanazawa University \smallskip\\
$^2$Japan Science and Technology Agency, PRESTO \smallskip\\
$^3$Department of Environment and System Sciences, Yokohama National University \smallskip\\
$^4$Graduate School of Natural Science and Technology, Kanazawa University
\smallskip\\
{\tt \{mkimura, notsu\}@se.kanazawa-u.ac.jp},\ \ 
{\tt ystanaka@ynu.ac.jp},\ \ 
{\tt mos@stu.kanazawa-u.ac.jp}
}
\begin{document}
\maketitle
\begin{abstract}
An extended Maxwell viscoelastic model with a relaxation parameter is studied from mathematical and numerical points of view.
It is shown that the model has a gradient flow property with respect to a viscoelastic energy.
Based on the gradient flow structure,
a structure-preserving time-discrete model is proposed and
existence of a unique solution is proved.
Moreover, a structure-preserving P1/P0 finite element scheme is presented
and its stability in the sense of energy is shown by using its discrete gradient flow structure.
As typical viscoelastic phenomena,
two-dimensional numerical examples by the proposed scheme
for a creep deformation and a stress relaxation
are shown and the effects of the relaxation parameter are investigated.
\par
\vspace{0.5em}
\noindent
\end{abstract}
%
%
%
%
%%%%%%%%%%%%%%%%%%%%%%%%%%%%%%%%%%%%%%%%%%%%%%%%%%%%%%%%%%%%%%%%%%%%%%
\section{Introduction}\label{sec:intro}
%%%%%%%%%%%%%%%%%%%%%%%%%%%%%%%%%%%%%%%%%%%%%%%%%%%%%%%%%%%%%%%%%%%%%%
In this paper, we develop a gradient flow structure of an extended Maxwell viscoelastic model, which naturally induces a stable structure-preserving P1/P0 finite element scheme.
The model includes a relaxation parameter~$\alpha \ge 0$ and it is
a variant of the standard linear solid model or the Zener model, see, e.g.,~\cite{AbuEbe-2007}.
We note that the model with $\alpha = 0$ is the well-known pure Maxwell viscoelastic model.
In that sense the model is an extension of the Maxwell viscoelastic model.
Although the argument to be presented in this paper can be applied to the so-called Zener(-type) models, here we focus on the extended Maxwell viscoelastic model.
Throughout the paper we shall often call this model simply the Maxwell model.
\par
There are many books and papers dealing with the Maxwell and other viscoelastic models.
For example, the books by Ferry~\cite{Ferry-1970}, Golden and Graham~\cite{GolGra-1988}, Lockett~\cite{Lockett-1972} and Macosko~\cite{Mac-1994}, and the papers by Karamanou et al.~\cite{KarShaWarWhi-2005}, Rivi{\`e}re and Shaw~\cite{RivSha-2006}, Rivi{\`e}re et al.~\cite{RivShaWheWhi-2003}, and Shaw and Whiteman~\cite{ShaWhi-2004}, where in the papers mainly discontinuous Galerkin finite element schemes have been proposed and analyzed.
Nevertheless, as far as we know, there are no papers discussing the gradient flow structure of the Maxwell or Zener viscoelastic models, which is important not only for the characterization of the model but also for the development of stable and convergent numerical schemes.
\par
It is easy to draw conceptual diagrams of the Maxwell viscoelastic model
with or without a sub-spring in one dimension, see, e.g.,
Figs.~\ref{fig:SLS_model}~($\alpha > 0$) or~\ref{fig:maxwell_model}~($\alpha = 0$), respectively.
For the Maxwell model with a sub-spring ($\alpha > 0$),
the system consists of an elastic spring and a subsystem connected in series.
This subsystem consists of a viscous dashpot and a sub-spring connected in parallel.
The pure Maxwell model is the Maxwell model without a sub-spring ($\alpha = 0$),
which consists of an elastic spring and a viscous dashpot connected in series.
According to the diagrams we can derive a $d$-dimensional extended Maxwell model $(d=2, 3)$.
It comprises of two equations; one is the force balance law of the system,
which can be seen as a quasi-static equation of linear elasticity,
and the other one is a time-dependent equation for the so-called viscosity effect.
\par
In this paper we present the gradient flow structure of the Maxwell model, which provides the energy decay property on the continuous level~(Theorem~\ref{thm:gradient_flow_continuous}).
As mentioned above, the structure is useful not only for the analysis but also for the development of stable and convergent numerical schemes.
Indeed, if the structure is also preserved for a numerical scheme on the discrete level, the stability of the scheme in the sense of energy follows in general.
\par
For a time-discrete Maxwell model, we present two results;
the existence and uniqueness of solutions~(Theorem~\ref{thm:solvability_tau}) and the corresponding discrete gradient flow structure and energy decay estimate~(Theorem~\ref{thm:gradient_flow_tau}).
\par
For the discretization in space we use the finite element method with P1/P0 element, i.e., the piecewise linear finite element~(P1-element) and the piecewise constant finite element~(P0-element) are employed for the approximations of the displacement of the viscoelastic body and the matrix-valued function for the viscosity effect, respectively.
We prove that the P1/P0 finite element scheme has a unique solution~(Theorem~\ref{thm:solvability_fem}) and that the scheme preserves the discrete gradient flow structure~(Theorem~\ref{thm:gradient_flow_fem}), which leads to the stability of the scheme in the sense of energy.
\par
The scheme is realized by an efficient algorithm~(see Algorithm on p.\pageref{AL}), where for each time-step the matrix-valued function is determined explicitly on each triangular element, although the backward Euler method is employed for the time integration.
Table~\ref{tab:eq_num_thm_num} lists the main results of this paper.
\begin{table}[!htbp]
\centering
\caption{Our main results.}
\label{tab:eq_num_thm_num}
\begin{tabular}{rccccc}
\toprule
& \multicolumn{2}{c}{Formulations} && \multicolumn{2}{c}{Results} \\ \cmidrule{2-3} \cmidrule{5-6}
& Strong & Weak && Exist. \& Uniq. & Gradient flow \\ \midrule
The continuous model & \eqref{prob:zener_strong} & \eqref{prob:maxwell_weak} && (in preparation~\cite{KNTY_prep}) & Theorem~\ref{thm:gradient_flow_continuous} \\
The time-discrete model & \eqref{prob:maxwell_timediscrete_strong} & \eqref{prob:maxwell_timediscrete_weak} && Theorem~\ref{thm:solvability_tau} & Theorem~\ref{thm:gradient_flow_tau}\\
The finite element scheme & -- & \eqref{scheme} && Theorem~\ref{thm:solvability_fem} & Theorem~\ref{thm:gradient_flow_fem}\\
\bottomrule
\end{tabular}
\end{table}
\par
We remark that the energy decay property of the Maxwell viscoelastic model was already shown in~\cite{RogWin-2010} and it was used to prove a stability estimate.
But they do not mention on its gradient flow structure in contrast to our approach.
In this paper, we prove the gradient flow structure with respect to a natural elastic energy and also propose a structure-preserving numerical scheme.
\par
To avoid the locking phenomena, stress formulations are often used in engineering.
Here we consider a displacement formulation and propose a structure-preserving P1/P0 finite element scheme, which may show the phenomena if the mesh size is not small enough.
The locking problem, however, can be overcome by an extended scheme with a pair of higher-order finite elements and/or adaptive mesh refinement technique~\cite{ALBERTA-2005} (see~Remark~\ref{rmk:high_order_schemes}-(ii) and -(iii)).
We note that the gradient flow structures in the continuous and discrete levels to be presented in this paper are the advantages of the displacement formulation.
\par
The paper is organized as follows.
In Section~\ref{sec:derivation} the governing equation of the extended Maxwell model
is derived and its initial and boundary value problem is stated.
In Section~\ref{sec:energy_decay} the gradient flow structure of the Maxwell model is presented and the energy decay estimate is shown.
In Section~\ref{sec:time_discretization} a time-discretization of the Maxwell model is studied;
the existence and uniqueness of solutions to the time-discretization of the Maxwell model is proved and the time-discrete gradient flow structure and the time-discrete energy decay estimate are shown.
In Section~\ref{sec:fem} a P1/P0 finite element scheme preserving the time-discrete gradient flow structure is presented.
In Section~\ref{sec:numerics} two-dimensional numerical results of the Maxwell model by the P1/P0 finite element scheme are shown.
%
%
%
%
%%%%%%%%%%%%%%%%%%%%%%%%%%%%%%%%%%%%%%%%%%%%%%%%%%%%%%%%%%%%%%%%%%%%%%
\section{The extended Maxwell model}\label{sec:derivation}
%%%%%%%%%%%%%%%%%%%%%%%%%%%%%%%%%%%%%%%%%%%%%%%%%%%%%%%%%%%%%%%%%%%%%%
%
%
%
%
The function spaces and the notation to be used throughout the paper are as follows.
Let $d = 2$ or $3$ be the dimension in space, and $\R^{d\times d}_\sym$ the space of symmetric $\R^{d\times d}$-valued matrices.
We suppose that $\Omega \subset \R^d$ is a bounded Lipschitz domain in this paper.
For a space~$R \in \{ \R^d, \R^{d\times d}_\sym\}$, the $R$-valued function spaces defined in $\Omega$ are denoted by $L^2(\Omega; R)$, $H^1(\Omega; R)$ and $C(\ol\Omega; R)$ etc.
For example, $H^1(\Omega; R)$ denotes the $R$-valued Sobolev space on $\Omega$.
For any normed space $X$ we define function spaces $C([0,t_0]; X)$ and $C^1([0,t_0]; X)$ consisting of $X$-valued functions in $C([0,t_0])$ and $C^1([0,t_0])$, respectively.
The dual pairing between $X$ and the dual space $X^\prime$ is denoted by ${}_{X^\prime}\lA\cdot, \cdot\rA_X$.
For normed spaces~$X$ and $Y$ the set of bounded linear operators from $X$ to $Y$ is denoted by $\mathcal{L}(X, Y)$.
For square matrices $A, B \in \R^{d\times d}$ we use the notation $A:B \, \defeq \sum_{i,j = 1}^d A_{ij} B_{ij}$.
%
%
%%%%%%%%%%%%%%%%%%%%%%%%%%%%%%%%%%%%%%%%%%%%%%%%%%%%%%%%%%%%%%%%%%%%%%
\subsection{Derivation of the model}
We introduce the extended Maxwell model with a relaxation term,
which is represented by a spring and a subsystem in series,
where the subsystem consists of a dashpot and a subspring connected in parallel,
cf. Fig.~\ref{fig:SLS_model}.
Let $e_0: \Omega\times [0, T) \to \R^{d\times d}_\sym$ and $\sigma_0: \Omega\times [0, T) \to \R^{d\times d}_\sym$ be the total strain and the total stress of a material governed by the Maxwell model.
Let $(e_i, \sigma_i): \Omega\times [0, T) \to \R^{d\times d}_\sym \times \R^{d\times d}_\sym$ for $i=1, \ldots, 4$ be the pairs of strain and stress for the left spring ($i=1$), the subsystem ($i=2$), the dashpot~($i=3$) and the subspring~($i=4$), respectively.
Furthermore, let $u = (u_1, \ldots, u_d)^T: \Omega \times [0, T) \to \R^d$ be the displacement of the material, $e[u]: \Omega\times [0, T) \to \R^{d\times d}_\sym$ the symmetric part of $\nabla u$ defined by
\begin{align}
e[u] \defeq \fz{1}{2} [ \nabla u +(\nabla u)^T ],
\label{def:e_u}
\end{align}
and $f: \Omega \times [0, T) \to \R^d$ be a given external force, where the superscript~$T$ denotes the transposition.
\begin{figure}[!htbp]
\centering
\includegraphics[bb=50 70 410 297,scale=0.4]{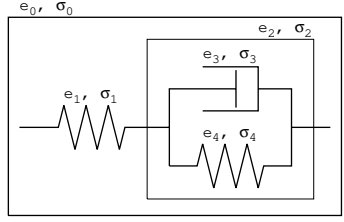}
\caption{A conceptual diagram of the Maxwell viscoelastic model with a relaxation term in one dimension.}
\label{fig:SLS_model}
\end{figure}
\par
According to Fig.~\ref{fig:SLS_model} we give the relations of $(e_i, \sigma_i)$ for $i=0, \ldots, 4$.
For the total strain and stress, $(e_0, \sigma_0)$, it is natural that the equations
\begin{align}
e_0 = e[u], \qquad - \nabla \cdot \sigma_0 = f
\label{eq:derivation_0}
\end{align}
hold; the former equation means that the total strain is expressed by $e[u]$, and the latter equation is the balance of forces.
We suppose that
\begin{subequations}\label{eqns:derivation}
\begin{align}
e_0 = e_1 + e_2, & \qquad \sigma_0 = \sigma_1 = \sigma_2, \label{eq:derivation_1} \\
\sigma_1 &= C e_1, \label{eq:derivation_2}
\end{align}
\end{subequations}
where $C=(c_{ijkl})_{i, j, k, l = 1, \ldots, d}$ is a fourth-order elasticity tensor for the left spring.
The series connection of spring and subsystem leads to~\eqref{eq:derivation_1},
and the Hooke's law implies~\eqref{eq:derivation_2}.
For the right subsystem we also suppose that
\begin{subequations}\label{eqns:derivation_zener}
\begin{align}
e_2 = e_3 = e_4, & \qquad \sigma_2 = \sigma_3 + \sigma_4, \label{eq:derivation_3_zener_1} \\
\eta \prz{e_3}{t} & = \sigma_3, \label{eq:derivation_3_zener_2} \\
\sigma_4 &= \alpha e_4, \label{eq:derivation_3_zener_3}
\end{align}
\end{subequations}
where $\eta>0$ is a viscosity constant of the material and $\alpha \ge 0$ is a scalar spring
constant of the subspring, which has a relaxation effect for the viscous dashpot.
In this paper, we call $\alpha$ a relaxation parameter.
The parallel connection of the dashpot and the subspring in the subsystem leads to~\eqref{eq:derivation_3_zener_1}.
The effect of the dashpot is taken into account through~\eqref{eq:derivation_3_zener_2} with~$\eta$.
Let us introduce the notation:
\begin{align}
\sigma[u,\phi] \defeq C(e[u] - \phi)\qquad \bigl( (u,\phi): \Omega \to \R^d \times \R^{d\times d}_\sym \bigr).
\label{def:sigma}
\end{align}
Then, we have the following relations,
\begin{align*}
e_1 & = e_0 - e_2 = e[u] - e_2,\qquad
\sigma_0 = \sigma_1 = Ce_1 = C(e[u] - e_2)=\sigma[u, e_2], \\
\eta \prz{e_2}{t} & = \eta \prz{e_3}{t} = \sigma_3 = \sigma_2 - \sigma_4 = \sigma_0 - \alpha e_2 = \sigma[u,e_2] - \alpha e_2,
\end{align*}
which yield
\begin{subequations}\label{eqns:derivation_summary_zener}
\begin{align}
-\nabla \cdot \sigma [u, e_2] & = f, \\
\eta \prz{e_2}{t} + \alpha e_2 & = \sigma [u, e_2].
\end{align}
\end{subequations}
This completes the derivation of the governing equations of the Maxwell model given by~\eqref{eqns:derivation_summary_zener}.
\begin{Rmk}
The system~\eqref{eqns:derivation_summary_zener} is an extension of the simplest {\rm Maxwell} model as shown in Fig.~\ref{fig:maxwell_model}, since \eqref{eqns:derivation_summary_zener} with $\alpha = 0$ leads to
\begin{subequations}\label{eqns:derivation_summary_maxwell}
\begin{align}
-\nabla \cdot \sigma [u, e_2] & = f, \\
\eta \prz{e_2}{t} & = \sigma [u, e_2],
\end{align}
\end{subequations}
which can be obtained similarly to the derivation of~\eqref{eqns:derivation_summary_zener} according to Fig.~\ref{fig:maxwell_model}.
\end{Rmk}
\begin{figure}[!htbp]
\centering
\includegraphics[bb=50 100 410 267,scale=0.4]{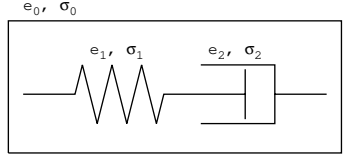}
\caption{A conceptual diagram of the Maxwell viscoelastic model without a relaxation term in one dimension.}
\label{fig:maxwell_model}
\end{figure}
%
%
%
%
%
%
%
%
%%%%%%%%%%%%%%%%%%%%%%%%%%%%%%%%%%%%%%%%%%%%%%%%%%%%%%%%%%%%%%%%%%%%%%
\subsection{Initial and boundary value problem}
We consider an initial and boundary value problem for the
extended Maxwell model.
The strain tensor variable $e_2$ in~\eqref{eqns:derivation_summary_zener}
is denoted by $\phi$ hereafter.
Let $\Gamma \defeq \pz\Omega$ be the boundary of~$\Omega$, and let $\Gamma_0$ be an open subset of~$\Gamma$ and $\Gamma_1 \defeq \Gamma \setminus \ol\Gamma_0$.
We suppose that the $(d-1)$-dimensional measure of~$\Gamma_0$ is positive and equal to that of~$\ol\Gamma_0$, where the case $\Gamma_0 = \Gamma$ ($\Gamma_1 = \emptyset$) is also available in the following.
\par
Now we summarize the mathematical formulation of the Maxwell model, which is to find $(u,\phi):\Omega \times (0,T) \to \R^d \times \R^{d \times d}_\sym$ such that
\begin{subequations}\label{prob:zener_strong}
\begin{align}
-\nabla \cdot \sigma[u,\phi] & = f & \mbox{in} & \ \ \Omega \times (0,T),
\label{prob:zener_strong_eq1} \\
 \eta \prz{\phi}{t} + \alpha\phi - \sigma[u,\phi] & = 0 & \mbox{in} & \ \ \Omega \times (0,T),
\label{prob:zener_strong_eq2} \\
 u & = g & \mbox{on} & \ \ \Gamma_0 \times (0,T),
\label{prob:zener_strong_eq3} \\
 \sigma[u,\phi]n & = q & \mbox{on} & \ \ \Gamma_1 \times (0,T),
\label{prob:zener_strong_eq4}\\
 \phi & = \phi^{0} & \mbox{in} & \ \ \Omega,\ \mbox{at $t=0$},
\label{prob:zener_strong_eq5}
\end{align}
\end{subequations}
where
$u:\Omega\times (0,T) \to \R^d$ is the displacement of the viscoelastic material,
$\phi:\Omega\times (0,T) \to \R^{d\times d}_{\rm sym}$ is the tensor describing the viscosity effect,
$\eta > 0$ and $\alpha \ge 0$ are given constants,
and $f: \Omega\times (0,T) \to \R^d$, $g: \Gamma_0\times (0,T) \to \R^d$, $q: \Gamma_1\times (0,T) \to \R^d$ and $\phi^0: \Omega\to \R^{d\times d}_\sym$ are given functions.
For the definitions of the stress tensor~$\sigma[u,\phi]$ and the strain tensor~$e[u]$, see~\eqref{def:sigma} and~\eqref{def:e_u}, respectively, where $C = (c_{ijkl})_{ijkl}$ used in the definition of $\sigma[u,\phi]$ is a given fourth-order elasticity tensor.
\par
In this paper, for simplicity, the next hypothesis is assumed to be held.
\begin{Hyp}\label{hyp:tensor_C}
(i)~The tensor $C$ is symmetric, isotropic and homogeneous, i.e.,
\begin{align}
c_{ijkl}(x) = c_{ijkl} = \lambda\delta_{ij}\delta_{kl} + \mu(\delta_{ik}\delta_{jl} + \delta_{il}\delta_{jk}), \quad \forall x \in \Omega,
\label{cond:elasticity_tensor}
\end{align}
for $\mu, \lambda \in \R$, where $\delta_{ij}$ is {\rm Kronecker}'s delta.
\smallskip \\
(ii)~The tensor $C$ is positive, i.e., there exists a positive constant~$c_\ast$ such that
\begin{align}
\sum_{i,j,k,l=1}^d c_{ijkl} \xi_{ij} \xi_{kl} \ge c_\ast |\xi|^2, \quad \forall \xi \in \R^{d\times d}_\sym,
\label{ieq:positivity}
\end{align}
where $|\xi| \defeq (\sum_{i,j=1}^d \xi_{ij}^2)^{1/2}$.
\end{Hyp}
\begin{Rmk}
(i)~$\lambda$ and $\mu$ are the so-called {\rm Lam\'e}'s constants. \\
(ii)~The positivity~\eqref{ieq:positivity} is satisfied for $c_\ast = 2\mu + \lambda d (> 0)$ if $\mu$ and $\lambda \in \R$ satisfy $\mu > 0$ and $\lambda > -(2/d) \mu$.
\end{Rmk}
%
%
%
%%%%%%%%%%%%%%%%%%%%%%%%%%%%%%%%%%%%%%%%%%%%%%%%%%%%%%%%%%%%%%%%%%%%%%
\subsection{Relationship between a viscoelastic model and the Maxwell model~\eqref{prob:zener_strong}}
Another viscoelastic model is well known and studied in~\cite{KarShaWarWhi-2005,RivSha-2006,RivShaWheWhi-2003,ShaWhi-2004}.
In the model the governing equations on the displacement $u(t) = u(\cdot, t):\Omega \to \R^d$ for $t \in (0,T)$ are represented as
\begin{subequations}\label{model:SW}
\begin{align}
-\nabla \cdot \sigma^{\rm Total}[u(t)] & = f(t) & \mbox{in} & \ \ \Omega, \\
\sigma^{\rm V}[u(t)] & = \int_0^t \prz{D}{s}(t-s) e[u(s)]~ds & \mbox{in} & \ \ \Omega,
\label{model:SW_eq2}
\end{align}
\end{subequations}
where $\sigma^{\rm Total}[u]$ is defined by
\[
\sigma^{\rm Total}[u] \defeq \sigma^{\rm E}[u] - \sigma^{\rm V}[u],
\qquad
\sigma^{\rm E}[u] \defeq D(0) e[u],
\]
and $D(t) = D(\cdot, t): \Omega \to \R^{d\times d\times d \times d}$ is a given fourth-order tensor.
The boundary conditions {in~\eqref{model:SW}} are omitted.
\par
For the sake of simplicity, we suppose that $D$ is homogeneous, and that
\begin{align*}
\prz{D}{t}(t) = e^{-\alpha t}C, \qquad 
D(0) = C,
\end{align*}
i.e., $D(t) = [1+(1/\alpha)(1-e^{-\alpha t})] C$.
Then, \eqref{model:SW_eq2} implies that
\begin{align*}
\prz{\sigma^{\rm V}}{t}[u(t)] + \alpha \sigma^{\rm V}[u(t)] & = C e[u(t)]  & \mbox{in} & \ \ \Omega.
\end{align*}
Multiplying both sides of the equation above on the left by~$C^{-1}$, letting $\phi^{\rm V}(t) \defeq C^{-1}\sigma^{\rm V}[u(t)]$ and noting that $\sigma^{\rm E}[u] = C e[u]$, we obtain
\begin{subequations}\label{model:SW_phi}
\begin{align}
-\nabla \cdot \sigma^{\rm Total}[u(t)] & = f(t) & \mbox{in} & \ \ \Omega, \\
\prz{\phi^{\rm V}}{t}(t) + \alpha \phi^{\rm V}(t) & = e[u(t)]  & \mbox{in} & \ \ \Omega,
\label{model:SW_phi_eq2}
\end{align}
\end{subequations}
with
\[
\sigma^{\rm Total}[u] = \sigma^{\rm E}[u] - C\phi^{\rm V} = C (e[u] - \phi^{\rm V}).
\]
The difference between~\eqref{prob:zener_strong}~(\eqref{prob:zener_strong_eq1}, \eqref{prob:zener_strong_eq2}) with~$\eta = 1$ and~\eqref{model:SW_phi} is in the second equations; in~\eqref{model:SW_phi_eq2} $e[u]$ is employed instead of~$\sigma[u,\phi]$ in~\eqref{prob:zener_strong_eq2}.
The gradient flow structure for~\eqref{model:SW_phi} for~$\alpha \ge 1$ can be derived similarly as in Section~\ref{sec:energy_decay}, cf.~Remark~\ref{rmk:energy_for_model_SW_phi} for details.
%
%
%%%%%%%%%%%%%%%%%%%%%%%%%%%%%%%%%%%%%%%%%%%%%%%%%%%%%%%%%%%%%%%%%%%%%%
\section{The gradient flow structure and the energy decay estimate}\label{sec:energy_decay}
%%%%%%%%%%%%%%%%%%%%%%%%%%%%%%%%%%%%%%%%%%%%%%%%%%%%%%%%%%%%%%%%%%%%%%
%
%
%
%
In this section we show the gradient flow structure and the energy decay estimate for the Maxwell model~\eqref{prob:zener_strong} after introducing a weak formulation of the model.
\par
We set a hypothesis for the given functions in model~\eqref{prob:zener_strong}.
\begin{Hyp}\label{hyp:given_funcs}
The given functions satisfy the following. \smallskip \\
(i)~$f \in C([0,T]; L^2(\Omega; \R^d))$, $g \in C([0,T]; H^1(\Omega; \R^d))$, $q\in C([0,T]; L^2(\Gamma_1; \R^d))$. \smallskip \\
(ii)~$\phi^0\in L^2(\Omega; \R^{d\times d}_\sym)$.
\end{Hyp}
\begin{Rmk}
It holds that $g(\cdot, t)_{|\Gamma_0} \in H^{1/2}(\Gamma_0; \R^d)$ from $g(\cdot, t) \in H^1(\Omega; \R^d)$ and the Trace Theorem~\cite{Nec-1967} for any $t\in [0, T]$.
\end{Rmk}
For a function $g_0 \in H^{1/2}(\Gamma_0;\R^d)$ let $X$, $V(g_0)$, $V$ and $\Psi$ be function spaces defined by
\begin{align*}
X \defeq H^1(\Omega;\R^d),
\ \
V(g_0) \defeq \bigl\{ v \in X;~v_{|\Gamma_0}=g_0 \bigr\},
\ \
V \defeq V(0),
\ \
\Psi \defeq L^2(\Omega;\R^{d\times d}_\sym).
\end{align*}
The inner product in $L^2(\Omega; \R^d)$ is denoted by $(\cdot, \cdot)$.
For the function space~$\Psi$ we use two inner products, $(\cdot, \cdot)_\Psi$ and $(\cdot, \cdot)_C$, defined by
\begin{align*}
(\phi, \psi)_\Psi & \defeq \int_\Omega \phi : \psi~dx,\quad (\phi, \psi)_C \defeq ( C\phi, \psi)_\Psi,
\end{align*}
which yield the norms~$\|\psi\|_\Psi \defeq (\psi, \psi)_\Psi^{1/2}$ and~$\|\psi\|_C \defeq (\psi, \psi)_C^{1/2}$, respectively.
\par
From the integration by parts, we obtain the weak formulation of model~\eqref{prob:zener_strong}; find $\{(u(t),\phi(t)) \in V(g(t)) \times \Psi;\ t\in (0,T) \}$ such that, for $t\in (0, T)$,
\begin{subequations}\label{prob:maxwell_weak}
\begin{align}
\bigl( \sigma[u(t), \phi(t)], e[v] \bigr)_\Psi & = \ell_t (v), & \forall v \in V,
\label{prob:maxwell_weak_eq1} \\
\eta \prz{\phi}{t}(t) + \alpha \phi(t) - \sigma[u(t),\phi(t)] & = 0 & \mbox{in}\ \Psi,
\label{prob:maxwell_weak_eq2}
\end{align}
\end{subequations}
with $\phi(0) = \phi^0$, where $\ell_t \in V^\prime$ is a linear form on $V$ defined by
\begin{align*}
\ell_t (v) & \defeq ( f(t), v ) + \int_{\Gamma_1} q(t) \cdot v~ds.
\end{align*}
\par
In the rest of Section~\ref{sec:energy_decay}, we suppose the condition:
\begin{align}
\prz{f}{t} = 0,
\quad
\prz{g}{t} = 0,
\quad
\prz{q}{t} = 0,
\label{cond:given_data_indep_of_time}
\end{align}
and that there exists a unique solution to~\eqref{prob:maxwell_weak}.
The linear form~$\ell_t$ is simply denoted by $\ell$ under~\eqref{cond:given_data_indep_of_time}.
\begin{Rmk}
Condition~\eqref{cond:given_data_indep_of_time} is not always assumed in the following sections.
In fact, condition~\eqref{cond:given_data_indep_of_time} is not assumed in Theorems~\ref{thm:solvability_tau} and~\ref{thm:solvability_fem},
while it is assumed in Theorems~\ref{thm:gradient_flow_continuous}, \ref{thm:gradient_flow_tau}, and~\ref{thm:gradient_flow_fem}.
\end{Rmk}
\par
We define an energy~$E (\cdot, \cdot): V(g)\times \Psi \to \R$ for model~\eqref{prob:zener_strong} by
\begin{align}
E (u, \phi) \defeq \fz{1}{2} \| e[u]-\phi \|_C^2 + \fz{\alpha}{2} \|\phi\|_\Psi^2 - \ell (u),
\end{align}
which has the following properties:
\begin{subequations}\label{eqns:derivative_E}
\begin{align}
(\pz_u E) (u, \phi) [v] & \defeq \fz{d}{d\varepsilon} E (u+\varepsilon v, \phi)_{|\varepsilon = 0} = (\sigma[u, \phi], e[v])_\Psi - \ell(v),
\label{eq:derivative_E_wrt_u}\\
(\pz_\phi E) (u, \phi) [\psi] & \defeq \fz{d}{d\varepsilon} E (u, \phi + \epsilon \psi)_{|\varepsilon = 0} = (\alpha \phi - \sigma[u, \phi], \psi)_\Psi,
\label{eq:derivative_E_wrt_phi}
\end{align}
\end{subequations}
for $u \in V(g)$, $v \in V$ and $\phi, \psi \in \Psi$.
We also define an energy $E_\ast: \Psi \to \R$ and its (G\^ateaux) derivative $(\pz E_\ast) (\phi) = (\pz E_\ast) (\phi) [\cdot]: \Psi \to \R$ by
\begin{subequations}\label{eqns:E_ast_etc}
\begin{align}
E_\ast (\psi) & \defeq \min_{v \in V(g)} E(v, \psi) = E(\ol{u} (\psi), \psi), & \psi & \in \Psi, \\
(\pz E_\ast) (\phi) [\psi] & \defeq \fz{d}{d\varepsilon} E_\ast (\phi+\varepsilon\psi)_{|\varepsilon = 0}, & \psi & \in \Psi,
\end{align}
\end{subequations}
where $\ol{u} (\psi) \in V(g)$ is the minimizer of~$E(v, \psi)$ defined by
\begin{align}
\ol{u} (\psi) \defeq \argmin_{v \in V(g)} E(v, \psi).
\label{def:bar_u}
\end{align}
\par
In the next theorem it is shown that the solution of~\eqref{prob:maxwell_weak} has a gradient flow structure under some assumptions.
\begin{Thm}[Gradient flow structure for the continuous model]\label{thm:gradient_flow_continuous}
Suppose that Hypotheses~\ref{hyp:tensor_C} and~\ref{hyp:given_funcs} and~\eqref{cond:given_data_indep_of_time} hold and that $(u, \phi) \in C^1([0,T]; X \times \Psi)$ is a solution to~\eqref{prob:maxwell_weak}.
Then, $(u, \phi)$ satisfies the following for any $t \in (0, T)$:
\smallskip\\
(i)~Gradient flow structure:
\begin{align}
\Bigl( \eta \prz{\phi}{t} (t), \psi \Bigr)_\Psi = - (\pz E_\ast) (\phi (t) ) [\psi],\quad \forall \psi \in \Psi.
\label{eq:gradient_flow_continuous}
\end{align}
(ii)~Energy decay estimate:
\begin{align}
\fz{d}{dt} E(u(t), \phi(t)) = -\eta \Bigl\| \prz{\phi}{t} (t) \Bigr\|_\Psi^2 \le 0.
\label{energy_decay_continuous}
\end{align}
\end{Thm}
\par
We prove the theorem after establishing the next two lemmas, where for $\phi \in \Psi$, the function~$\ol{u}(\phi)$ defined in~\eqref{def:bar_u} and an operator $(\pz \ol{u}) (\phi)$ defined by
\begin{align}
(\pz \ol{u}) (\phi) [\psi] & \defeq \fz{d}{d\varepsilon} \ol{u} (\phi+\varepsilon \psi)_{|\varepsilon = 0},\quad \forall \psi \in \Psi,
\label{eq:derivative_bar_u}
\end{align}
are studied.
\begin{Lem}\label{lem:derivative_bar_u}
Suppose that Hypotheses~\ref{hyp:tensor_C} and~\ref{hyp:given_funcs}-(i) and~\eqref{cond:given_data_indep_of_time} hold.
Then, for any $\phi\in \Psi$ the function $\ol{u}(\phi) \in V(g)$ is well defined.
Moreover, there exists a unique operator $(\pz \ol{u}) (\phi) \in \mathcal{L}(\Psi, V)$ such that
\[
(\pz \ol{u}) (\phi) = A^{-1}B,
\]
where $A \in \mathcal{L}(V, V^\prime)$ and $B \in \mathcal{L}(\Psi, V^\prime)$, and they are defined in~\eqref{def:A_B}.
\end{Lem}
\begin{proof}
From~\eqref{eq:derivative_E_wrt_u}, if $\ol{u} \in V(g)$ is a minimizer of $E(\cdot,\phi)$, $\ol{u}$ satisfies
\begin{align}\label{proof}
\ell(v) = ( \sigma[\ol{u}, \phi], e[v] )_\Psi = (e[\ol{u} - \phi], e[v])_C, \qquad \forall v \in V.
\end{align}
Setting $\tilde{u} \defeq \ol{u} - g \in V$, we rewrite (\ref{proof}) as
\begin{align*}
A\tilde{u} = B \phi + \tilde{\ell} \quad \mbox{in $V^\prime$,}
\end{align*}
where $A \in \mathcal{L}(V, V^\prime)$, $B \in \mathcal{L}(\Psi, V^\prime)$ and $\tilde{\ell} \in V^\prime$ are defined by
\begin{align}
{}_{V^\prime} \lA A\tilde{u}, v \rA_{V} & \defeq (e[\tilde{u}], e[v] )_C,
\qquad {}_{V^\prime} \lA B \phi, v \rA_{V} \defeq (\phi, e[v] )_C, \label{def:A_B}\\
{}_{V^\prime}\lA \tilde{\ell}, v \rA_{V} & \defeq \ell (v) - (e[g], e[v] )_C. \notag
\end{align}
From the positivity of~$C$, i.e., \eqref{ieq:positivity}, and the Lax--Milgram Theorem, cf., e.g., \cite{Cia-1978},  
$A^{-1} \in \mathcal{L}(V',V)$ holds. Then there exists a unique $\tilde{u} = \tilde{u}(\phi) = A^{-1}(B\phi + \tilde{\ell}) \in V$, which implies that the unique solution of (\ref{proof}) is given by
\begin{align}
\ol{u} = \ol{u}(\phi) := A^{-1}(B\phi + \tilde{\ell}) + g \in V(g).
\label{proof2}
\end{align}
For arbitrary $v \in V(g)$ we have
\begin{align*}
E(v,\phi) - E(\ol{u},\phi) &= \fz{1}{2} \| e[v] - \phi \|^2_C - \fz{1}{2}\| e[\ol{u} - \phi] \|^2_C - \ell(v - \ol{u}) \\
&= \fz{1}{2} (e[v+\ol{u}] - 2\phi, e[v - \ol{u}])_C - (e[\ol{u}] - \phi, e[v - \ol{u}])_C \\
&= \fz{1}{2} \| e[v - \ol{u}]\|^2_C \ge 0 .
\end{align*}
This shows that $\ol{u} = \ol{u}(\phi)$ is the unique minimizer of $E(v,\phi)$ among $v \in V(g)$.
Hence, from \eqref{proof2}, we also conclude that
\begin{align*}
\pz \ol{u}(\phi) = A^{-1}B \in \mathcal{L}(\Psi,V).
\end{align*}
\end{proof}
\begin{Lem}\label{lem:gradient_E_ast}
For $\phi, \psi \in \Psi$, it holds that
\begin{align}
(\pz E_\ast) (\phi) [\psi] = (\alpha \phi - \sigma[\ol{u}(\phi),\phi], \psi)_\Psi.
\end{align}
\end{Lem}
\begin{proof}
Using~\eqref{eq:derivative_E_wrt_u}, \eqref{def:bar_u} and Lemma~\ref{lem:derivative_bar_u} and noting that $(\pz_u E) (\ol{u} (\psi), \psi) [v] = 0$ holds for any $v \in V$ and $\psi \in \Psi$,
we have
\begin{align*}
(\pz E_\ast) (\phi) [\psi] & = \fz{d}{d\varepsilon} E_\ast (\phi + \epsilon \psi)_{|\varepsilon = 0} = \fz{d}{d\varepsilon} E(\ol{u} (\phi + \epsilon \psi), \phi + \epsilon \psi)_{|\varepsilon = 0} \\
& = \Bigl[ (\pz_u E) (\ol{u} (\phi + \epsilon \psi), \phi + \epsilon \psi) [ (\pz \ol{u}) (\phi + \epsilon \psi) [\psi] ]
+ (\pz_\phi E) (\ol{u} (\phi + \epsilon \psi), \phi + \epsilon \psi) [\psi] \Bigr]_{|\varepsilon = 0} \\
& = (\pz_u E) (\ol{u} (\phi), \phi) [ (\pz \ol{u}) (\phi) [\psi] ]
+ (\pz_\phi E) (\ol{u} (\phi), \phi) [\psi] \\
& = (\pz_\phi E) (\ol{u} (\phi), \phi) [\psi] 
= (\alpha \phi - \sigma[\ol{u}(\phi), \phi], \psi)_\Psi,
\end{align*}
which completes the proof.
\end{proof}
\par
From this lemma, Theorem~\ref{thm:gradient_flow_continuous} is shown as follows.
\medskip\\
{\it Proof of Theorem~\ref{thm:gradient_flow_continuous}.}
\ \ 
Let $t \in (0, T)$ be fixed arbitrarily.
We omit ``$(t)$'' from $u(t)$ and $\phi(t)$ whenever convenient.
Since $(u(t), \phi(t))$ is a solution to~\eqref{prob:maxwell_weak}, we obtain from~\eqref{prob:maxwell_weak_eq1}, \eqref{eq:derivative_E_wrt_u} and \eqref{def:bar_u} that $u(t) = \ol{u}(\phi(t))$.
Together with Lemma~\ref{lem:gradient_E_ast} and~\eqref{prob:maxwell_weak_eq1} we have
\begin{align}
(\pz E_\ast) (\phi) [\psi] 
= (\alpha \phi - \sigma[u, \phi], \psi)_\Psi
= \Bigl( -\eta \prz{\phi}{t}, \psi \Bigr)_\Psi
\label{eq:gradient_flow_continuous_a}
\end{align}
for any $\psi \in \Psi$.
Hence \eqref{eq:gradient_flow_continuous} holds.
\par
From~\eqref{eqns:derivative_E} and~\eqref{eq:gradient_flow_continuous_a} we have
\begin{align*}
\fz{d}{dt} E(u, \phi) & = (\pz_u E)(u, \phi)\Bigl[ \prz{u}{t} \Bigr] + (\pz_\phi E)(u, \phi) \Bigl[ \prz{\phi}{t} \Bigr] = (\pz_\phi E)(u, \phi) \Bigl[ \prz{\phi}{t} \Bigr] \\
& = \Bigl( \alpha \phi - \sigma[u, \phi], \prz{\phi}{t} \Bigr)_\Psi = -\eta \Bigl\| \prz{\phi}{t} \Bigr\|_\Psi^2 \le 0,
\end{align*}
which implies~\eqref{energy_decay_continuous}.
\qed
\begin{Rmk}\label{rmk:energy_for_model_SW_phi}
In the case of model~\eqref{model:SW_phi}, a corresponding energy is defined by
\begin{align*}
\tilde{E}[u,\phi^\mathrm{V}]
\defeq \fz{1}{2} \| e[u] - \phi^\mathrm{V} \|_C^2 + \fz{\alpha -1}{2} \|\phi^\mathrm{V}\|_C^2 - \ell(u),
\end{align*}
where it is natural to consider $\alpha \ge 1$ for non-negative energy~$\tilde{E}$.
%we note that $\alpha \ge 1$ is natural to have non-negative $\tilde{E}$.
%
Letting
\begin{align*}
\tilde{E}_\ast (\phi^\mathrm{V}) \defeq \min_{u\in V(g)}\tilde{E} (u, \phi^\mathrm{V}),
\qquad
(\pz\tilde{E}_\ast) ( \phi^\mathrm{V} )[\psi] \defeq \fz{d}{d\varepsilon} \tilde{E}_\ast (\phi^\mathrm{V}+\varepsilon\psi)_{|\varepsilon = 0},
\end{align*}
similarly to~\eqref{eqns:E_ast_etc}, we have the following.
\smallskip\\
(i)~Gradient flow structure:
\begin{align*}
\Bigl( \prz{\phi^\mathrm{V}}{t}(t), \psi \Bigr)_C = - (\pz\tilde{E}_\ast) ( \phi^\mathrm{V} (t) )[\psi], \quad \forall \psi \in \Psi.
\end{align*}
(ii)~Energy decay estimate:
\begin{align*}
\fz{d}{dt} \tilde{E}[u(t),\phi^\mathrm{V}(t)] = - \biggl\| \prz{\phi^\mathrm{V}}{t} \biggr\|_C^2 \le 0.
\end{align*}
\end{Rmk}
%
%
%%%%%%%%%%%%%%%%%%%%%%%%%%%%%%%%%%%%%%%%%%%%%%%%%%%%%%%%%%%%%%%%%%%%%%
\section{The time-discrete Maxwell model}\label{sec:time_discretization}
%%%%%%%%%%%%%%%%%%%%%%%%%%%%%%%%%%%%%%%%%%%%%%%%%%%%%%%%%%%%%%%%%%%%%%
%
%
%
%
%%%%%%%%%%%%%%%%%%%%%%%%%%%%%%%%%%%%%%%%%%%%%%%%%%%%%%%%%%%%
\subsection{Existence and uniqueness for the time-discrete model}
We discretize the Maxwell model~\eqref{prob:zener_strong} in time.
Let $\tau > 0$ be a time increment, and let $N_T \defeq \lfloor T/\tau \rfloor$ and $t^k \defeq k\tau$ for $k=0,\ldots, N_T$.
In the following we set $\rho^k\defeq \rho(\cdot, t^k)$ for a function $\rho$ defined in $\Omega \times (0,T)$ or on $\Gamma_i \times (0,T)$, $i=0, 1$.
The time-discrete problem for~\eqref{prob:zener_strong} is to find $\{(u_\tau^k,\phi_\tau^k):\Omega \to \R^d \times \R^{d\times d}_\sym ;\ k=0, \ldots, N_T\}$ such that
\begin{subequations}\label{prob:maxwell_timediscrete_strong}
\begin{align}
-\nabla \cdot \sigma[u_\tau^k,\phi_\tau^k] & = f^k & \mbox{in} & \ \ \Omega, \quad k=0, \ldots, N_T, \\
  \eta \ol{D}_\tau \phi_\tau^k + \alpha\phi_\tau^k - \sigma[u_\tau^k,\phi_\tau^k] & = 0 & \mbox{in} & \ \ \Omega, \quad k=1, \ldots, N_T, \\
  u_\tau^k & = g^k & \mbox{on} & \ \ \Gamma_0, \quad k=0, \ldots, N_T, \\
  \sigma[u_\tau^k,\phi_\tau^k]n & = q^k & \mbox{on} & \ \ \Gamma_1, \quad k=0, \ldots, N_T, \\
  \phi_\tau^0 & = \phi^0 & \mbox{in} & \ \ \Omega,
\end{align}
\end{subequations}
where $\ol{D}_\tau$ is the backward difference operator $\ol{D}_\tau \rho^k \defeq (\rho^k-\rho^{k-1})/\tau$.
\par
From the integration by parts, we get the weak formulation of~\eqref{prob:maxwell_timediscrete_strong}; find $\{(u_\tau^k,\phi_\tau^k)\in V(g^k) \times \Psi;\ k=0,\ldots, N_T \}$ such that
\begin{subequations}\label{prob:maxwell_timediscrete_weak}
\begin{align}
\bigl( \sigma[u_\tau^k,\phi_\tau^k], e[v] \bigr)_\Psi & = \ell^k(v), & \forall v \in V, \quad k & = 0, \ldots, N_T,
\label{prob:maxwell_timediscrete_weak_eq1} \\
\eta \ol{D}_\tau \phi_\tau^k + \alpha \phi_\tau^k - \sigma[u_\tau^k,\phi_\tau^k] &= 0 & \mbox{in} \ \Psi, \quad k & = 1, \ldots, N_T,
\label{prob:maxwell_timediscrete_weak_eq2}
\end{align}
\end{subequations}
with $\phi_\tau^0 = \phi^0$, where $\ell^k \in V^\prime$ is a linear form on $V$ defined by, for $k=0,\ldots,N_T$,
\begin{align*}
\ell^k (v) & \defeq ( f^k, v ) + \int_{\Gamma_1} q^k \cdot v~ds.
\end{align*}
\par
In the next theorem we state and prove the uniqueness and existence of solutions to~\eqref{prob:maxwell_timediscrete_weak} from the Lax--Milgram Theorem.
%
%
%
%
%
%%%%%%%
\begin{Thm}[Existence and uniqueness for the time-discrete model]\label{thm:solvability_tau}
Suppose that Hypotheses~\ref{hyp:tensor_C} and~\ref{hyp:given_funcs} hold.
Then, there exists a unique solution~$\{(u_\tau^k,\phi_\tau^k)\in V(g^k) \times \Psi;\ k=0,\ldots, N_T \}$ to~\eqref{prob:maxwell_timediscrete_weak}.
\end{Thm}
%%%%%%%
%
%
\begin{proof}
Since $\phi_\tau^0 (= \phi^0) \in \Psi$ is known, there exists a unique solution~$(u_\tau^0, \phi_\tau^0) \in V(g^0) \times \Psi$ of~\eqref{prob:maxwell_timediscrete_weak_eq1} with $k=0$ from the positivity of $C$ and the Lax--Milgram Theorem.
\par
We show the existence of solutions to~\eqref{prob:maxwell_timediscrete_weak} ($k \ge 1$) by induction.
Supposing that $\phi_\tau^{k-1} \in \Psi$ is given for a fixed $k\in \{1,\ldots,N_T\}$, we show that there exists a solution $(u_\tau^k, \phi_\tau^k) \in V(g^k) \times \Psi$ to~\eqref{prob:maxwell_timediscrete_weak}.
The equation~\eqref{prob:maxwell_timediscrete_weak_eq2} yields an explicit representation of~$\phi_\tau^k$,
\begin{align}
\phi_\tau^k = D^{-1} \Bigl( Ce[u_\tau^k] + \fz{\eta}{\tau} \phi_\tau^{k-1} \Bigr),
\label{prob:maxwell_timediscrete_weak_eq2_explicit}
\end{align}
where $D$ is a fourth-order tensor defined by
\begin{align}
D \defeq \Bigl( \fz{\eta}{\tau} + \alpha \Bigr) I + C
\label{def:tensor_D}
\end{align}
for the (fourth-order) identity tensor~$I$ with $I_{ijkl} \defeq \delta_{ik} \delta_{jl}$.
Substituting~\eqref{prob:maxwell_timediscrete_weak_eq2_explicit} into~\eqref{prob:maxwell_timediscrete_weak_eq1}, we have
\begin{align}
\Bigl( C(I+D^{-1}C) e[u_\tau^k], e[v] \Bigr)_\Psi = \Bigl( \fz{\eta}{\tau} CD^{-1} \phi_\tau^{k-1}, e[v] \Bigr)_\Psi + \ell^k (v), \quad \forall v \in V.
\label{prob:maxwell_timediscrete_weak_eq1_without_phi_k}
\end{align}
We note that \eqref{prob:maxwell_timediscrete_weak_eq1_without_phi_k} can be seen as a system of linear elasticity with a positive elasticity tensor~$C(I+D^{-1}C)$.
From the Lax--Milgram Theorem we have the uniqueness and existence of $u_\tau^k \in V(g^k)$ to~\eqref{prob:maxwell_timediscrete_weak_eq1_without_phi_k}.
We obtain $\phi_\tau^k \in \Psi$ from~\eqref{prob:maxwell_timediscrete_weak_eq2_explicit}.
It is obvious that $(u_\tau^k, \phi_\tau^k) \in V(g^k) \times \Psi$ satisfies~\eqref{prob:maxwell_timediscrete_weak}.
Thus, we find a solution~$\{(u_\tau^k,\phi_\tau^k)\in V(g^k) \times \Psi;\ k=1,\ldots, N_T \}$ to~\eqref{prob:maxwell_timediscrete_weak} inductively.
\par
Next we show the uniqueness.
$(u_\tau^0, \phi_\tau^0) \in V(g^0)\times \Psi$ is uniquely determined as mentioned in the beginning of the proof.
By linearity, we assume without loss of generality that $f^k = 0$, $q^k = 0$, $g^k = 0$ and $\phi_\tau^{k-1} = 0$ for any $k \in \{1,\ldots,N_T\}$.
We show that $(u_\tau^k, \phi_\tau^k) = (0, 0) \in V\times \Psi$.
From~\eqref{prob:maxwell_timediscrete_weak_eq1_without_phi_k} and $\ell^k=0$ we have $u_\tau^k = 0 \in V$, which yields from~\eqref{prob:maxwell_timediscrete_weak_eq2} that $\phi_\tau^k = 0 \in \Psi$.
\end{proof}
\subsection{Gradient flow structure and energy decay estimate for the time-discrete model}
The Maxwell model~\eqref{prob:zener_strong} has the gradient flow structure~\eqref{eq:gradient_flow_continuous}.
Here, we present a time-discrete version of the gradient flow structure~\eqref{eq:gradient_flow_continuous} and an energy decay estimate for the solution of~\eqref{prob:maxwell_timediscrete_weak}, which is a discrete version of~\eqref{energy_decay_continuous}.
Let $E_\tau^k \defeq E (u_\tau^k, \phi_\tau^k)$ for the solution $\{(u_\tau^k,\phi_\tau^k)\}_{k=1}^{N_T}$ of~\eqref{prob:maxwell_timediscrete_weak}.
When condition~\eqref{cond:given_data_indep_of_time} holds true, i.e., $\ell^k$ is independent of $k$, we omit the superscript~$k$ from $\ell^k$.
\begin{Thm}[Gradient flow structure for the time-discrete model]\label{thm:gradient_flow_tau}
Suppose that Hypotheses~\ref{hyp:tensor_C} and~\ref{hyp:given_funcs} and~\eqref{cond:given_data_indep_of_time} hold.
Let $\{(u_\tau^k,\phi_\tau^k)\in V(g) \times \Psi;\ k=0,\ldots, N_T \}$ be the solution of~\eqref{prob:maxwell_timediscrete_weak}.
Then, the solution satisfies the following for any $k=1,\ldots, N_T$:
\smallskip\\
(i)~Gradient flow structure:
\begin{align}
\bigl( \eta \ol{D}_\tau \phi_\tau^k, \psi \bigr)_\Psi = - (\pz E_\ast) (\phi_\tau^k) [\psi], \quad \forall \psi \in \Psi.
\label{eq:gradient_flow_tau}
\end{align}
(ii)~Energy decay estimate:
\begin{align}
\ol{D}_\tau E_\tau^k + \fz{\alpha \tau}{2} \bigl\|\ol{D}_\tau \phi_\tau^k \bigr\|_\Psi^2 + \fz{\tau}{2} \bigl\| \ol{D}_\tau (e [ u_\tau^k ] - \phi_\tau^k) \bigr\|_C^2 = -\eta \bigl\| \ol{D}_\tau \phi_\tau^k \bigr\|_\Psi^2 \le 0.
\label{eq:energy_decay_tau}
\end{align}
\end{Thm}
\begin{proof}
Let $k \in \{1,\ldots,N_T\}$ be fixed arbitrarily.
Since $(u_\tau^k, \phi_\tau^k)$ is the solution to~\eqref{prob:maxwell_timediscrete_weak}, we obtain from~\eqref{prob:maxwell_timediscrete_weak_eq1}, \eqref{eq:derivative_E_wrt_u} and \eqref{def:bar_u} that $u_\tau^k = \ol{u}(\phi_\tau^k)$.
Together with Lemma~\ref{lem:gradient_E_ast} and~\eqref{prob:maxwell_timediscrete_weak_eq2} we have
\begin{align}
(\pz E_\ast) (\phi_\tau^k) [\psi]
= (\alpha \phi_\tau^k - \sigma[u_\tau^k, \phi_\tau^k], \psi)_\Psi
= ( -\eta \ol{D}_\tau \phi_\tau^k, \psi )_\Psi
\label{eq:gradient_flow_tau_a}
\end{align}
for any $\psi \in \Psi$.
Hence~\eqref{eq:gradient_flow_tau} holds.
\par
The estimate~\eqref{eq:energy_decay_tau} is proved as follows.
Using the identity $a^2-b^2 = (a+b)(a-b)$, $u_\tau^k - u_\tau^{k-1} = \tau \ol{D}_\tau u_\tau^k$ and $\phi_\tau^k - \phi_\tau^{k-1} = \tau \ol{D}_\tau \phi_\tau^k$, we have
\begin{align}
\ol{D}_\tau E_\tau^k
& = \fz{1}{\tau} \Bigl[ E(u_\tau^k, \phi_\tau^k) - E(u_\tau^{k-1}, \phi_\tau^{k-1}) \Bigr]\notag\\
& = \fz{1}{2} \Bigl( \sigma[u_\tau^k+u_\tau^{k-1}, \phi_\tau^k+\phi_\tau^{k-1}], e[ \ol{D}_\tau u_\tau^k ] - \ol{D}_\tau \phi_\tau^k \Bigr)_\Psi 
+ \fz{\alpha}{2} \Bigl( \phi_\tau^k + \phi_\tau^{k-1}, \ol{D}_\tau \phi_\tau^k \Bigr)_\Psi - \ell ( \ol{D}_\tau u_\tau^k) \notag\\
& = \Bigl( \sigma[u_\tau^k, \phi_\tau^k], e[\ol{D}_\tau u_\tau^k] - \ol{D}_\tau \phi_\tau^k \Bigr)_\Psi - \fz{\tau}{2} \Bigl( \sigma[\ol{D}_\tau u_\tau^k, \ol{D}_\tau \phi_\tau^k ], e[\ol{D}_\tau u_\tau^k] - \ol{D}_\tau \phi_\tau^k \Bigr)_\Psi \notag\\
& \quad + \alpha \Bigl( \phi_\tau^k, \ol{D}_\tau \phi_\tau^k \Bigr)_\Psi - \fz{\alpha\tau}{2} \Bigl( \ol{D}_\tau \phi_\tau^k, \ol{D}_\tau \phi_\tau^k \Bigr)_\Psi - \ell (\ol{D}_\tau u_\tau^k) \notag \\
& = \Bigl( \sigma[u_\tau^k, \phi_\tau^k], e[\ol{D}_\tau u_\tau^k] \Bigr)_\Psi - \ell (\ol{D}_\tau u_\tau^k) + \Bigl( \alpha \phi_\tau^k - \sigma[u_\tau^k, \phi_\tau^k], \ol{D}_\tau \phi_\tau^k \Bigr)_\Psi
- \fz{\tau}{2} \bigl\| e[\ol{D}_\tau u_\tau^k] - \ol{D}_\tau \phi_\tau^k \bigr\|_C^2 - \fz{\alpha\tau}{2} \bigl\| \ol{D}_\tau \phi_\tau^k \bigr\|_\Psi^2 \notag \\
& = - \eta \bigl\| \ol{D}_\tau \phi_\tau^k \bigr\|_\Psi^2 - \fz{\tau}{2} \bigl\| e[\ol{D}_\tau u_\tau^k] - \ol{D}_\tau \phi_\tau^k \bigr\|_C^2 - \fz{\alpha\tau}{2} \bigl\| \ol{D}_\tau \phi_\tau^k \bigr\|_\Psi^2,
\label{eq:energy_proof}
\end{align}
which implies~\eqref{eq:energy_decay_tau}, 
where \eqref{prob:maxwell_timediscrete_weak} with $v = \ol{D}_\tau u_\tau^k \in V$ and $\psi = \ol{D}_\tau \phi_\tau^k \in \Psi$ and~\eqref{eq:gradient_flow_tau_a} have been employed for the last equality in~\eqref{eq:energy_proof}.
\end{proof}
\begin{Cor}[Energy decay estimate for the time-discrete model]
Under the same assumptions of Theorem~\ref{thm:gradient_flow_tau}, it holds that
\begin{align}
E_\tau^k \le E_\tau^{k-1}, \quad \forall k=1,\ldots,N_T.
\label{ieq:energy_decay_tau}
\end{align}
\end{Cor}
\begin{proof}
From~\eqref{eq:energy_decay_tau} we have $\ol{D}_\tau E_\tau^k \le - \eta \bigl\|\ol{D}_\tau \phi_\tau^k \bigr\|_\Psi^2 \le 0$
which yields~\eqref{ieq:energy_decay_tau}.
\end{proof}
%
%
%
%
%
%%%%%%%%%%%%%%%%%%%%%%%%%%%%%%%%%%%%%%%%%%%%%%%%%%%%%%%%%%%%
\section{A P1/P0 finite element scheme}\label{sec:fem}
%%%%%%%%%%%%%%%%%%%%%%%%%%%%%%%%%%%%%%%%%%%%%%%%%%%%%%%%%%%%
%
In this section we present a finite element scheme for the Maxwell model~\eqref{prob:zener_strong} and show a gradient flow structure for the discrete system to be given by~\eqref{scheme}.
\subsection{A finite element scheme with an efficient algorithm}
Let $\mathcal{T}_h=\{K\}$ be a triangulation of $\Omega$, where $h$ is a representative size of the triangular elements, and let $\Omega_h = {\rm int} (\cup_{K\in\mathcal{T}_h} K)$.
For the sake of simplicity, we assume $\Omega=\Omega_h$.
We define finite element spaces~$X_h$ and~$\Psi_h$ by
\begin{align*}
X_h & \defeq \bigr\{ v_h \in C(\ol\Omega; \R^d);\ v_{h|K} \in P_1(K; \R^d), \forall K \in \mathcal{T}_h \bigr\}, \\
\Psi_h & \defeq \bigr\{ \psi_h \in L^2(\ol\Omega; \R^{d\times d}_\sym);\ \psi_{h|K} \in P_0(K; \R^{d\times d}), \forall K \in \mathcal{T}_h \bigr\},
\end{align*}
where $P_1(K; \R^d)$ and $P_0(K; \R^{d\times d})$ are polynomial spaces of vector-valued linear functions and matrix-valued constant functions on $K\in\mathcal{T}_h$, respectively.
For a function~$g_{0h}\in X_h$ we define function spaces $V_h(g_{0h})$ and $V_h$ by $V_h(g_{0h}) \defeq X_h \cap V(g_{0h})$ and $V_h \defeq V_h(0)$, respectively.
\par
Suppose that $\{g_h^k\}_{k=0}^{N_T} \subset X_h$ and $\phi_h^0 \in \Psi_h$ are given, where $g_h^k$ and $\phi_h^0$ are approximations of $g^k$ and $\phi^0$, respectively.
We present a finite element scheme for the Maxwell model~\eqref{prob:zener_strong}; find $\{(u_h^k,\phi_h^k)\in V_h(g_h^k)\times \Psi_h; \ k=0, \ldots, N_T \}$ such that
\begin{subequations}\label{scheme}
\begin{align}
\bigl( \sigma[u_h^k,\phi_h^k], e[v_h] \bigr)_\Psi & = \ell^k(v_h), & \forall v_h \in V_h, \quad & k=0, \ldots, N_T,
\label{scheme_eq1} \\
\eta \ol{D}_\tau \phi_h^k + \alpha \phi_h^k - \sigma[u_h^k,\phi_h^k] & = 0 & \mbox{in $\Psi_h$}, \quad & k=1, \ldots, N_T.
\label{scheme_eq2}
\end{align}
\end{subequations}
\par
Thanks to the choice of P1/P0-finite element for $X_h \times \Psi_h$, we have that $e[u_h^k]_{|K} \in P_0(K; \R^{d\times d})$ for any $K\in\mathcal{T}_h$, and the equation~\eqref{scheme_eq2} can be considered on each $K$.
Similarly to~\eqref{prob:maxwell_timediscrete_weak_eq2_explicit} the equation~\eqref{scheme_eq2} provides an explicit representation of~$\phi_h^k$,
\begin{align}
\phi_h^k = D^{-1} \Bigl( Ce[u_h^k] + \fz{\eta}{\tau} \phi_h^{k-1} \Bigr),
\label{scheme_eq2_explicit}
\end{align}
where $D$ is the tensor defined in~\eqref{def:tensor_D}.
Substituting~\eqref{scheme_eq2_explicit} into~\eqref{scheme_eq1}, we have
\begin{align}
\Bigl( C(I+D^{-1}C) e[u_h^k], e[v_h] \Bigr)_\Psi = \Bigl( \fz{\eta}{\tau} CD^{-1} \phi_h^{k-1}, e[v_h] \Bigr)_\Psi + \ell^k (v_h), \quad \forall v_h \in V_h.
\label{scheme_eq1_new}
\end{align}
Hence, scheme~\eqref{scheme} is realized by the next algorithm for $k \ge 1$, while $u_h^0 \in V_h(g_h^0)$ is obtained from~\eqref{scheme_eq1} with $k=0$.
\begin{Alg}\label{algorithm}
Let a function $\phi_h^{k-1} \in \Psi_h$ be given for some $k\in \{1, \ldots, N_T\}$.
Then, the pair $(u_h^k, \phi_h^k) \in V_h(g_h^k) \times \Psi_h$ is obtained as follows:
\label{AL}
\begin{enumerate}
\item
Find $u_h^k \in V_h(g_h^k)$ by~\eqref{scheme_eq1_new}, which is a symmetric system of linear equations. \smallskip
\item
Find $\phi_h^k \in \Psi_h$ by~\eqref{scheme_eq2_explicit}, where $\phi_h^k$ is determined explicitly on each triangular element.
\end{enumerate}
\end{Alg}
\begin{Rmk}
For given continuous functions~$g \in C(\ol\Omega\times [0,T]; \R^d)$ and~$\phi^0 \in C(\ol\Omega; \R^{d\times d}_\sym)$ we define
\begin{align}
g_h^k \defeq g_h(\cdot, t^k) \defeq \Pi_h^{(1)} g (\cdot, t^k) \in X_h, \qquad \phi_h^0 \defeq \Pi_h^{(0)} \phi^0 \in \Psi_h,
\end{align}
for $k = 0, \ldots, N_T$, where $\Pi_h^{(0)}:~C(\ol\Omega; \R^{d\times d}_\sym) \to \Psi_h$ and $\Pi_h^{(1)}:~C(\ol\Omega; \R^d) \to X_h$ are the {\rm Lagrange} interpolation operators.
\end{Rmk}
\begin{Rmk}
We note that
\begin{align*}
  D^{-1} X =  \fz{1}{\beta_0} \Bigl[ X - \fz{\lambda}{\beta_1} (\mathrm{tr}\,X) I \Bigr],
  \quad \forall X \in \R^{d\times d}_\sym ,
\end{align*}
where $\beta_0 \defeq 2\mu + (\eta/\tau) + \alpha$, $\beta_1 \defeq d\lambda + \beta_0$,
and $I \in \R^{d\times d}$ is the identity matrix.
\end{Rmk}
\begin{Rmk}\label{rmk:high_order_schemes}
(i)~In the case of a conforming pair, e.g.,  P2/P1 element, we have to solve a linear system to determine the function $\phi_h^k$ due to the continuity of $\phi_h^k$.\\
(ii)~A similar algorithm is possible for the pair of continuous P$\ell$ and discontinuous P${\ell^\prime}$ finite element spaces~(P$\ell$/P$\ell^\prime$dc), $\ell \in \N$, $\ell^\prime = \ell-1$, for $u$ and $\phi$, respectively.\\
(iii)~Since the locking phenomena often happen for P1-FEM in the displacement formulation, the stress formulation is often used, for example~\cite{RogWin-2010}.
But the locking problem can be avoided by using P2/P1dc element and/or adaptive mesh refinement technique~\cite{ALBERTA-2005}, and the gradient structures of our continuous and discrete models are huge advantages of the displacement formulation.
\end{Rmk}
\par
The next theorem shows on the existence and uniqueness of the solutions to~\eqref{scheme}.
%
%
%%%%%%%
\begin{Thm}[Existence and uniqueness for the finite element scheme]\label{thm:solvability_fem}
Suppose that Hypotheses~\ref{hyp:tensor_C} and~\ref{hyp:given_funcs} hold and that $\{g_h^k\}_{k=1}^{N_T} \subset X_h$ and $\phi_h^0 \in \Psi_h$ are given.
Then, there exists a unique solution~$\{(u_h^k,\phi_h^k)\in V_h(g_h^k) \times \Psi_h;\ k=0,\ldots, N_T \}$ to~\eqref{scheme}.
\end{Thm}
%%%%%%%
%
%
\begin{proof}
We show the existence of a solution to~\eqref{scheme}.
The proof is similar to that of Theorem~\ref{thm:solvability_tau}.
Since $\phi_h^0 \in \Psi_h$ is known, there exists a unique solution~$(u_h^0, \phi_h^0) \in V_h(g_h^0) \times \Psi_h$ of~\eqref{scheme_eq1} with $k=0$ from the Lax--Milgram Theorem.
\par
Let $k \in \{1,\ldots, N_T\}$ be fixed arbitrarily and $\phi_h^{k-1} \in \Psi_h$ be given.
From~\eqref{scheme_eq2} we obtain~\eqref{scheme_eq2_explicit} and~\eqref{scheme_eq1_new}, which imply the existence of $(u_h^k,\phi_h^k)\in V_h(g_h^k) \times \Psi_h$ from the Lax--Milgram Theorem.
We omit the proof of the uniqueness, since it is similar to that of Theorem~\ref{thm:solvability_tau}.
\end{proof}
\subsection{Gradient flow structure and energy decay estimate for the finite element scheme}
\label{subsec:fem_energy}
We assume Hypothesis~\ref{hyp:given_funcs} in the rest of this section.
Let $E_h^k \defeq E (u_h^k, \phi_h^k)$ for the finite element solution $\{(u_h^k,\phi_h^k)\}_{k=1}^{N_T}$ of~\eqref{scheme}.
For $\phi_h \in \Psi_h$ we also define an energy $E_{h \ast}: \Psi_h \to \R$ and its (G\^ateaux) derivative $(\pz E_{h \ast}) (\phi_h) = (\pz E_{h \ast}) (\phi_h) [\cdot]: \Psi_h \to \R$ by
\begin{align*}
E_{h \ast} (\psi_h) & \defeq \min_{v_h \in V_h(g_h)} E (v_h, \psi_h) = E (\ol{u}_h(\psi_h), \psi_h), && \psi_h \in \Psi_h, \\
(\pz E_{h \ast}) (\phi_h) [\psi_h] & \defeq \fz{d}{d\varepsilon} E_{h \ast} (\phi_h+\varepsilon\psi_h)_{|\varepsilon = 0}, && \psi_h \in \Psi_h.
\end{align*}
where $\ol{u}_h(\psi_h) \in V_h(g_h)$ is the minimizer of~$E (v, \psi)$ defined by
\begin{align}
\ol{u}_h(\psi_h) \defeq \argmin_{v_h \in V_h(g_h)} E (v_h, \psi_h).
\label{def:bar_u_h}
\end{align}
\par
We present a gradient flow structure and an energy decay estimate for the scheme in~\eqref{scheme} which is the discrete counterpart of~\eqref{eq:gradient_flow_tau} and~\eqref{eq:energy_decay_tau} in Theorem~\ref{thm:gradient_flow_tau}.
\begin{Thm}[Gradient flow structure for the finite element scheme]\label{thm:gradient_flow_fem}
Suppose that Hypotheses~\ref{hyp:tensor_C} and~\ref{hyp:given_funcs} and~\eqref{cond:given_data_indep_of_time} hold and that $\{g_h^k\}_{k=0}^{N_T} \subset X_h$ and $\phi_h^0 \in \Psi_h$ are given with $g_h^k$ independent of $k$.
Let $\{(u_h^k,\phi_h^k)\in V_h(g_h) \times \Psi;\ k=0,\ldots, N_T \}$ be the solution of~\eqref{scheme}.
Then, the solution satisfies the following for any $k=1,\ldots, N_T$:
\smallskip\\
(i)~Gradient flow structure:
\begin{align}
\bigl( \eta \ol{D}_\tau \phi_h^k, \psi_h \bigr)_\Psi = - (\pz E_{h \ast}) (\phi_h^k) [\psi_h], \quad \forall \psi_h \in \Psi_h.
\label{eq:gradient_flow_fem}
\end{align}
(ii)~Energy decay estimate:
\begin{align}
\ol{D}_\tau E_h^k + \fz{\alpha \tau}{2} \bigl\|\ol{D}_\tau \phi_h^k \bigr\|_\Psi^2 + \fz{\tau}{2} \bigl\| \ol{D}_\tau (e [ u_h^k ] - \phi_h^k) \bigr\|_C^2 = -\eta \bigl\| \ol{D}_\tau \phi_h^k \bigr\|_\Psi^2 \le 0.
\label{eq:energy_decay_fem}
\end{align}
\end{Thm}
\begin{proof}
The proof is similar to that of Theorem~\ref{thm:gradient_flow_tau}.
Let $k \in \{1,\ldots,N_T\}$ be fixed arbitrarily.
Since $(u_h^k, \phi_h^k)$ is the solution to~\eqref{scheme}, we obtain from~\eqref{scheme_eq1}, \eqref{eq:derivative_E_wrt_u} and \eqref{def:bar_u} that $u_h^k = \ol{u}_h(\phi_h^k)$.
Together with Lemma~\ref{lem:gradient_E_ast} and~\eqref{scheme_eq2} we have
\begin{align}
& (\pz E_{h \ast}) (\phi_h^k) [\psi_h] = \fz{d}{d\varepsilon} E_{h \ast} (\phi_h^k + \epsilon \psi_h)_{|\varepsilon = 0} = \fz{d}{d\varepsilon} E(\ol{u}_h (\phi_h^k + \epsilon \psi_h), \phi_h^k + \epsilon \psi_h)_{|\varepsilon = 0} \notag\\
& = (\pz_\phi E) (\ol{u}_h (\phi_h^k), \phi_h^k) [\psi_h] 
= (\alpha \phi_h^k - \sigma[u_h^k, \phi_h^k], \psi_h)_\Psi
= ( -\eta \ol{D}_\tau \phi_h^k, \psi_h )_\Psi
\label{eq:gradient_flow_fem_a}
\end{align}
for any $\psi_h \in \Psi_h$.
Hence~\eqref{eq:gradient_flow_fem} holds.
\par
From \eqref{scheme} with $v_h = \ol{D}_\tau u_h^k \in V_h$ and $\psi_h = \ol{D}_\tau \phi_h^k \in \Psi_h$ and~\eqref{eq:gradient_flow_fem_a}, we have
\begin{align}
\ol{D}_\tau E_h^k
& = \fz{1}{\tau} \Bigl[ E(u_h^k, \phi_h^k) - E(u_h^{k-1}, \phi_h^{k-1}) \Bigr]\notag\\
& = \Bigl( \sigma[u_h^k, \phi_h^k], e[\ol{D}_\tau u_h^k] \Bigr)_\Psi - \ell (\ol{D}_\tau u_h^k) + \Bigl( \alpha \phi_h^k - \sigma[u_h^k, \phi_h^k], \ol{D}_\tau \phi_h^k \Bigr)_\Psi 
%\notag\\
%& \quad 
- \fz{\tau}{2} \bigl\| e[\ol{D}_\tau u_h^k] - \ol{D}_\tau \phi_h^k \bigr\|_C^2 - \fz{\alpha\tau}{2} \bigl\| \ol{D}_\tau \phi_h^k \bigr\|_\Psi^2 \quad \mbox{(cf. \eqref{eq:energy_proof})} \notag \\
& = - \eta \bigl\| \ol{D}_\tau \phi_h^k \bigr\|_\Psi^2 - \fz{\tau}{2} \bigl\| e[\ol{D}_\tau u_h^k] - \ol{D}_\tau \phi_h^k \bigr\|_C^2 - \fz{\alpha\tau}{2} \bigl\| \ol{D}_\tau \phi_h^k \bigr\|_\Psi^2, \notag
\end{align}
which implies~\eqref{eq:energy_decay_fem}.
\end{proof}
\begin{Cor}[Energy decay estimate for the finite element scheme]
Under the same assumptions in Theorem~\ref{thm:gradient_flow_fem} it holds that
\begin{align}
E_h^k \le E_h^{k-1}, \quad \forall k=1,\ldots,N_T.
\label{ieq:energy_decay_fem}
\end{align}
\end{Cor}
\begin{proof}
From~\eqref{eq:energy_decay_fem} we have $\ol{D}_\tau E_h^k \le - \eta \bigl\|\ol{D}_\tau \phi_h^k \bigr\|_\Psi^2 \le 0$, which implies~\eqref{ieq:energy_decay_fem}.
\end{proof}
%
%
%
%
%
%
%%%%%%%%%%%%%%%%%%%%%%%%%%%%%%%%%%%%%%%%%%%%%%%%%%%%%%%%%%%%
\section{Numerical results}\label{sec:numerics}
%%%%%%%%%%%%%%%%%%%%%%%%%%%%%%%%%%%%%%%%%%%%%%%%%%%%%%%%%%%%
%
In this section numerical results in 2D for two examples below are presented, where we set
\[
\Omega = (0, 1)^2, \quad \lambda = \mu = \eta = 1, \quad q=0, \quad \phi^0 = 0.
\]
To observe the effect of the relaxation parameter~$\alpha$ we use three values of $\alpha$,
\[
\alpha=0, 1, 2.
\]
The examples are solved by the scheme in~\eqref{scheme} with $\tau = 0.01$ and a non-uniform mesh generated by FreeFem++~\cite{FreeFem} as shown in Fig.~\ref{fig:mesh}, where the division number of each side of the domain is~$40$, i.e., $h = 1/40$.
The total number of elements is $3,794$ and the total number of nodes is $1,978$.
\begin{figure}[!htbp]
\centering
\includegraphics[bb=100 0 496 400,scale=.3]{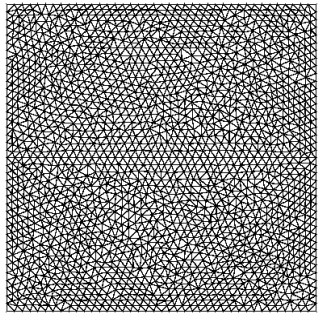}
\caption{The mesh used in the computation.}
\label{fig:mesh}
\end{figure}
\begin{Ex}\label{ex:1}
Let $\Gamma_0=\{x \in \pz\Omega;\ x_1\in (0,1),\ x_2=1 \}$, 
$T=1$, $f=(0,-1)^T$, $g=(0, 0)^T$.
\end{Ex}
\begin{Ex}\label{ex:2}
Let $\Gamma_0=\{x \in \pz\Omega;\ x_1 = 0, 1,\ x_2 \in (0, 1) \}$, 
$T=2$, $f=(0, 0)^T$, $g=(x_1, 0)^T$.
\end{Ex}
\par
The first example is solved in order to see a typical viscoelastic phenomenon, {\it creep}, where $f=(0, -1)^T$ corresponds to the gravity force acting on the viscoelastic body, and the top lid is fixed.
Fig.~\ref{fig:time_evolution_ex1} shows a time evolution of the shape of the material for $\alpha = 0$~(left), $1$~(center) and~$2$~(right), and Fig.~\ref{fig:energy_ex1} illustrates the energy as a function of time for the three values of $\alpha$.
We observe that the square domain has been expanded gradually depending on the value of $\alpha$ and that the energy decay property (see Theorem~\ref{thm:gradient_flow_fem}) is realized numerically.
It is well known that the creep behavior of real materials cannot be predicted
by the pure Maxwell model~($\alpha=0$). As shown in the left column of Fig.~\ref{fig:time_evolution_ex1}
and  a solid line in Fig.~\ref{fig:energy_ex1},
the displacement and the strain increases 
and the elastic energy decreases both almost linearly in time.
However, most of the viscoelastic materials such as polymers behave
not linearly under constant load but have some certain bounds of the
displacement and the elastic energy as shown in the cases $\alpha=1,~2$.
\par
We solve the second example to observe another typical viscoelastic phenomenon known as {\it stress relaxation.}
We test for different values of~$\alpha$.
Here we simply impose $u=g$ on $\ol\Gamma_0$ for $t > 0$, while $\phi = 0$ at $t=0$.
Similarly to the case of Example~\ref{ex:1}, Fig.~\ref{fig:time_evolution_ex2} shows a time evolution of the shape of the material for $\alpha = 0$~(left), $1$~(center) and~$2$~(right), and Fig.~\ref{fig:energy_ex2} illustrates the energy as a function of time for the three values of $\alpha$.
In the case of $\alpha = 0$, the shapes of top and bottom lids of the deformed domain are almost flat at $t=2$.
On the other hand, in the case of $\alpha = 1$ and~$2$, we can see the curved top and bottom lids at $t=2$, which  are the effect of relaxation parameter~$\alpha$.
Fig.~\ref{fig:stress_relax} shows $\|\sigma_{11} [u_h^k,\phi_h^k] \|_{L^\infty(\Omega)}~(k=1, \ldots, N_T)$ as a function of time for $\alpha = 0$, $1$ and~$2$, where the stress relaxation with respect to time~$t$ is observed.
We observe that, in the case of $\alpha = 0$, the stress $\|\sigma_{11}(t)\|_{L^\infty(\Omega)}$ goes to zero as $t$ increases, and that, in the cases of $\alpha = 1$ and~$2$, it goes approximately to $0.73$ and~$1.15$, respectively, which are the effect of the relaxation parameter~$\alpha$.
\begin{figure}[!htbp]
\centering
(a0)~\includegraphics[bb=100 50 360 302,scale=.36]{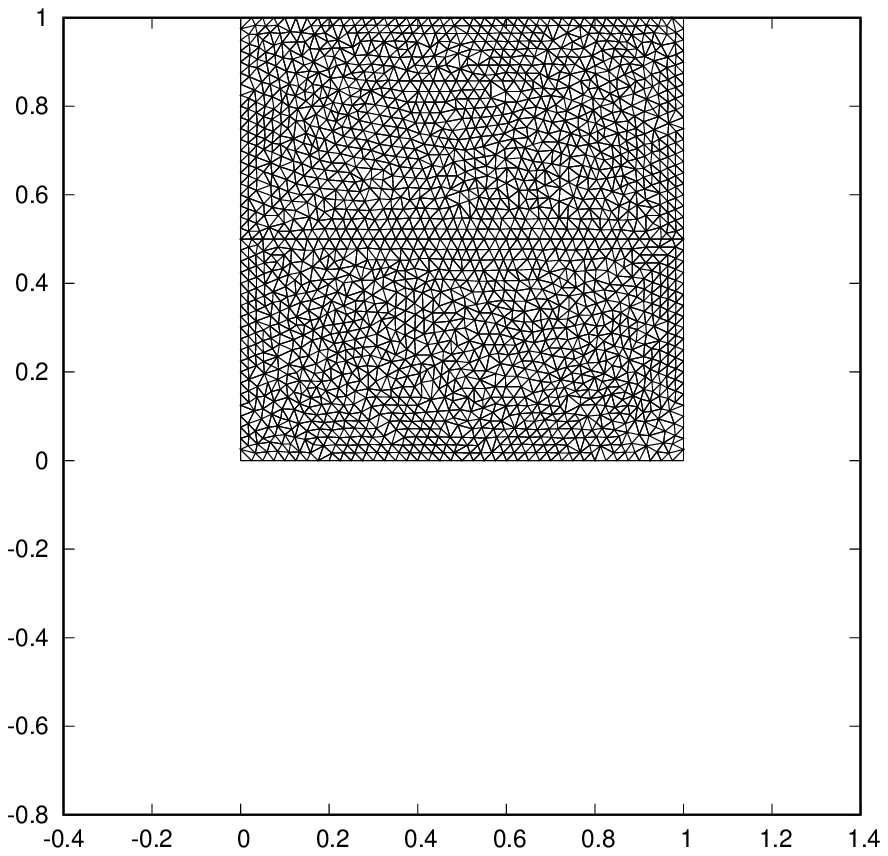}
(b0)~\includegraphics[bb=100 50 360 302,scale=.36]{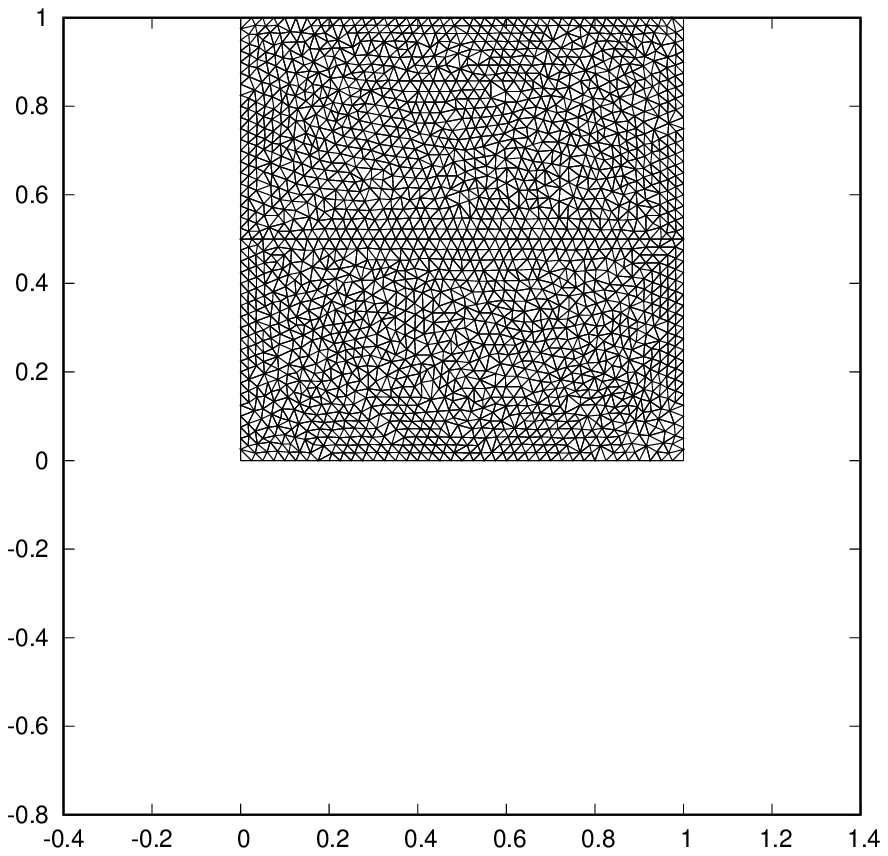}
(c0)~\includegraphics[bb=100 50 360 302,scale=.36]{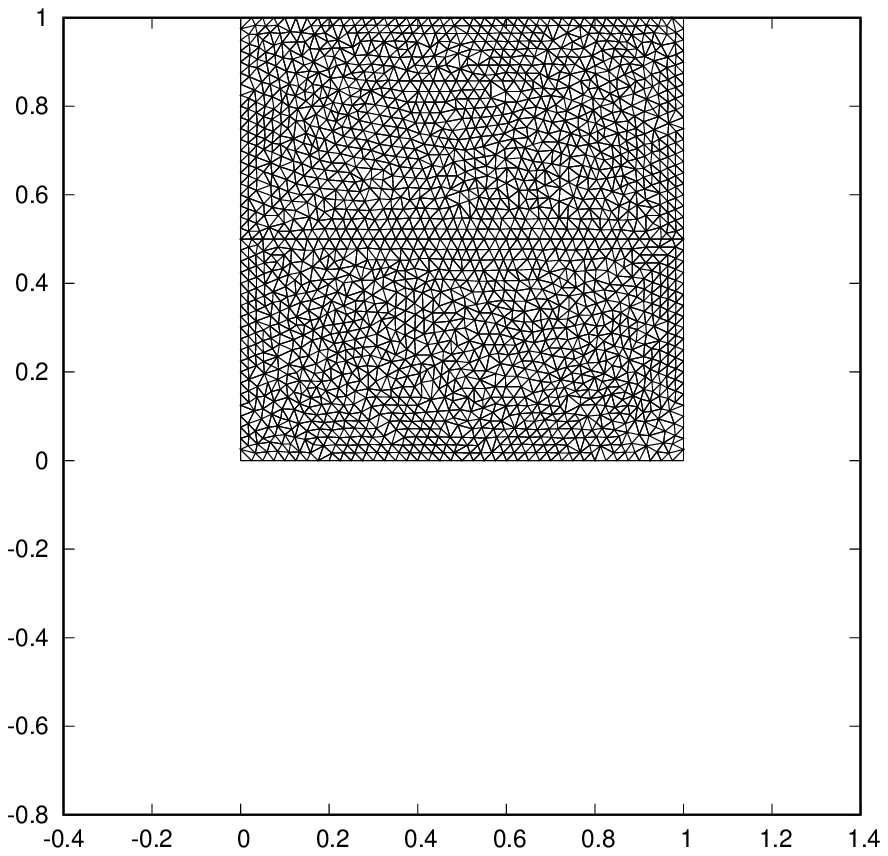}
\bigskip\\
(a1)~\includegraphics[bb=100 50 360 302,scale=.36]{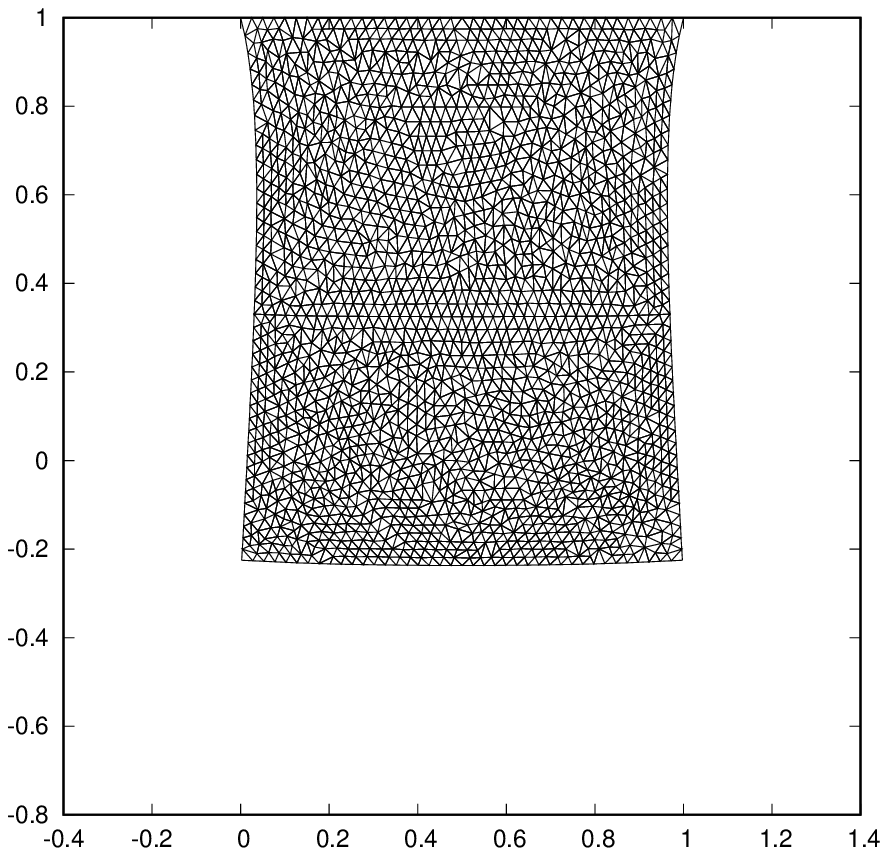}
(b1)~\includegraphics[bb=100 50 360 302,scale=.36]{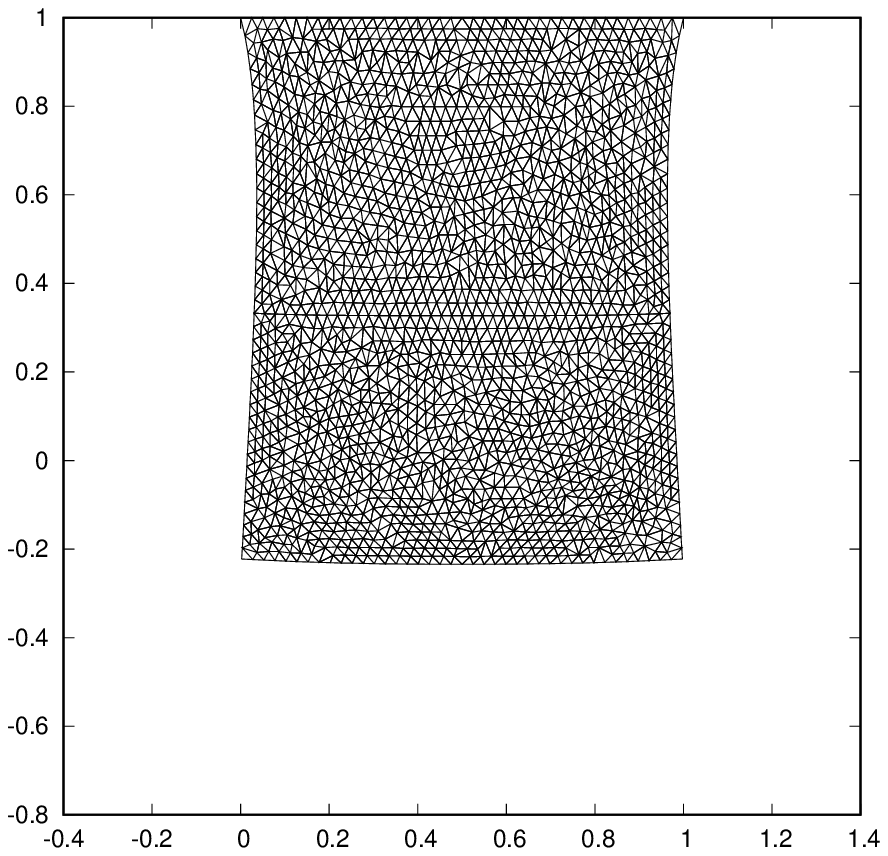}
(c1)~\includegraphics[bb=100 50 360 302,scale=.36]{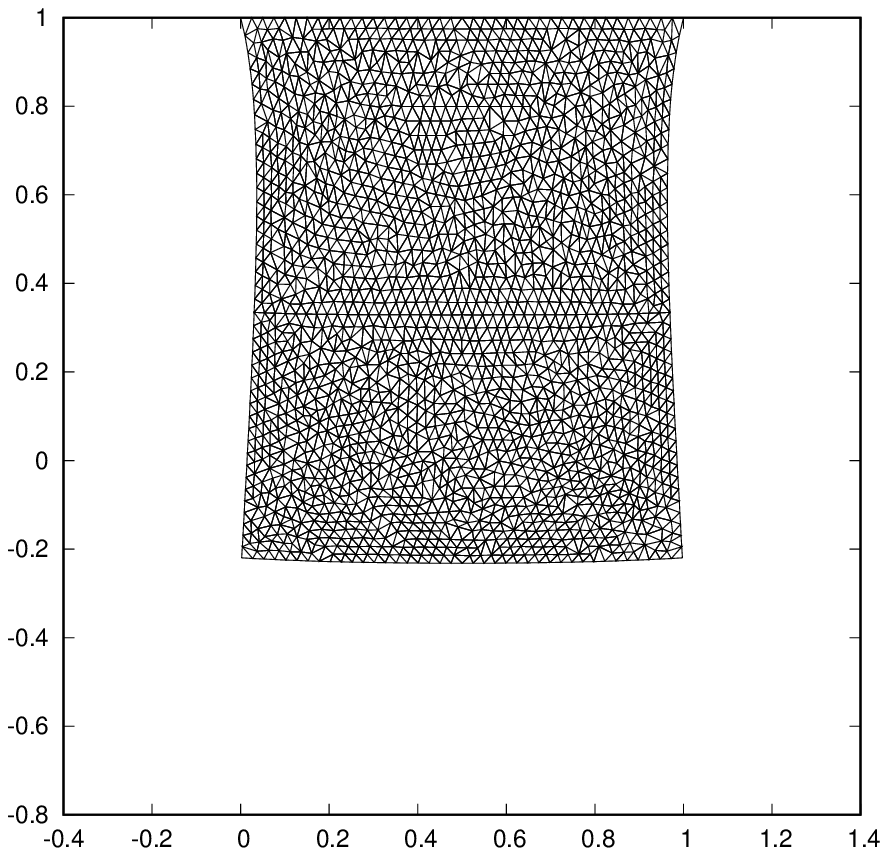}
\bigskip\\
(a2)~\includegraphics[bb=100 50 360 302,scale=.36]{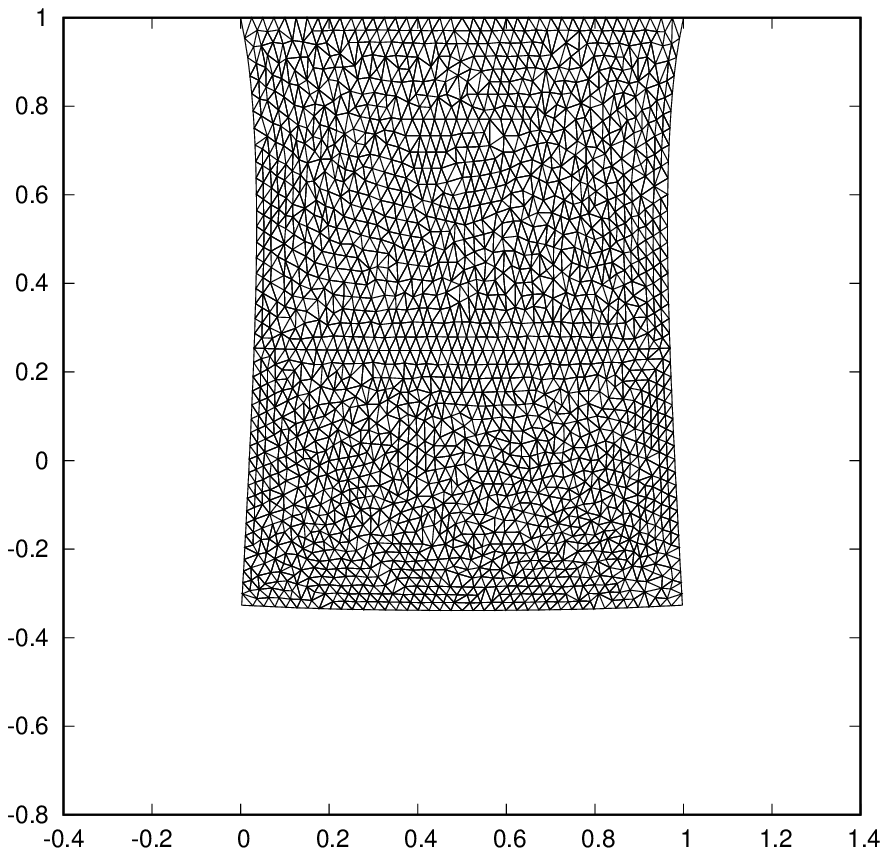}
(b2)~\includegraphics[bb=100 50 360 302,scale=.36]{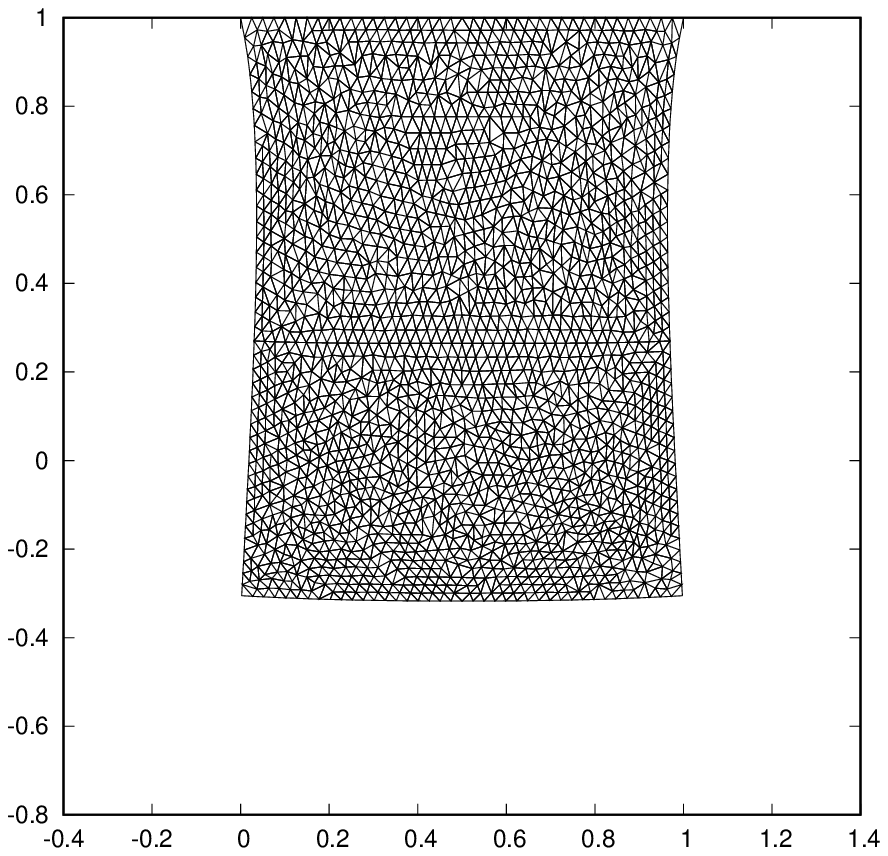}
(c2)~\includegraphics[bb=100 50 360 302,scale=.36]{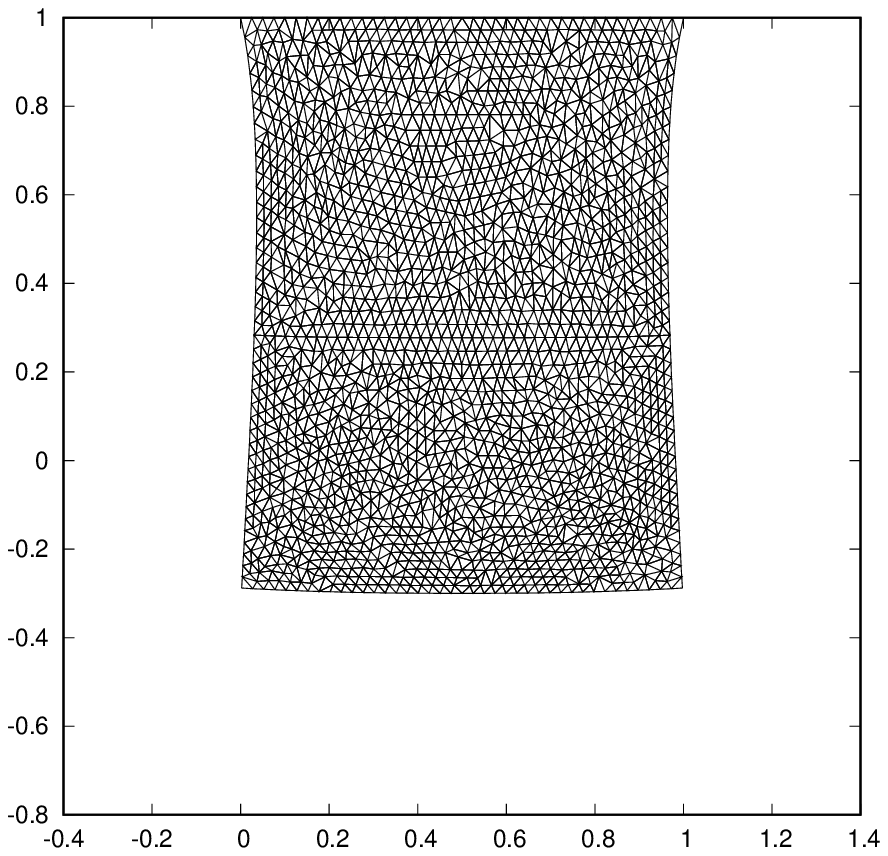}
\bigskip\\
(a3)~\includegraphics[bb=100 50 360 302,scale=.36]{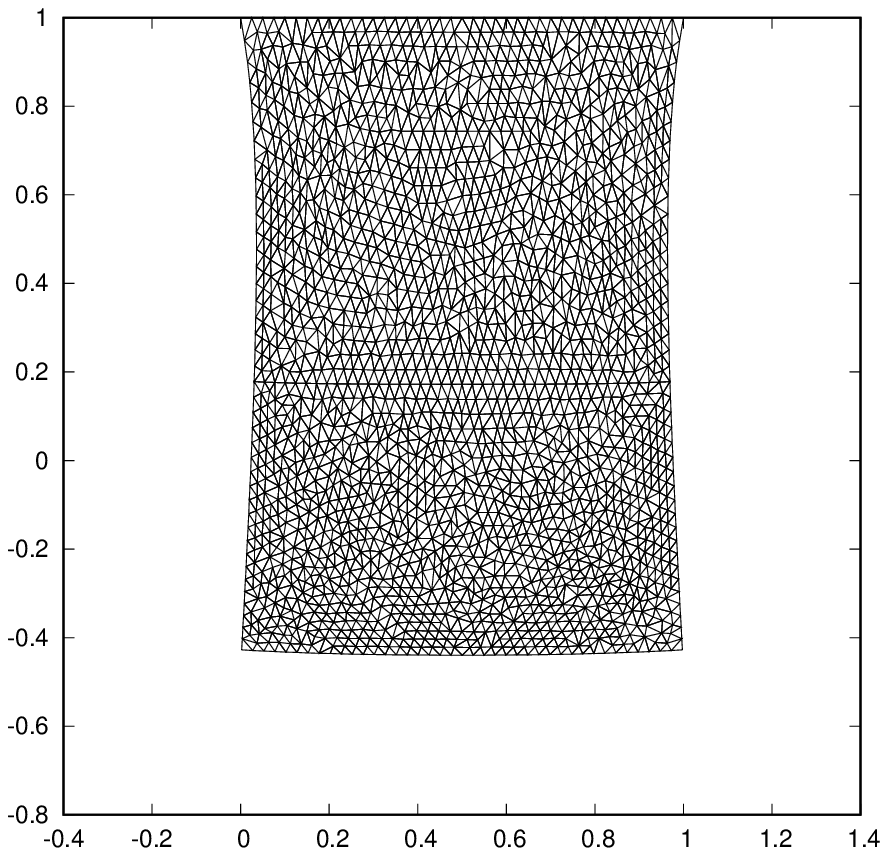}
(b3)~\includegraphics[bb=100 50 360 302,scale=.36]{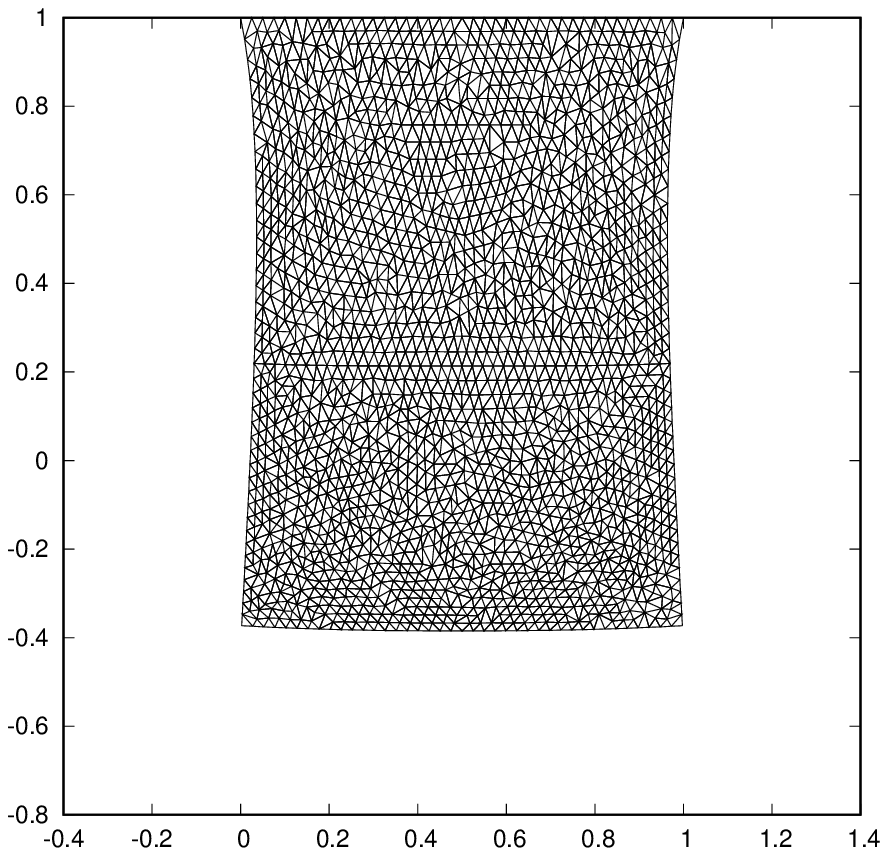}
(c3)~\includegraphics[bb=100 50 360 302,scale=.36]{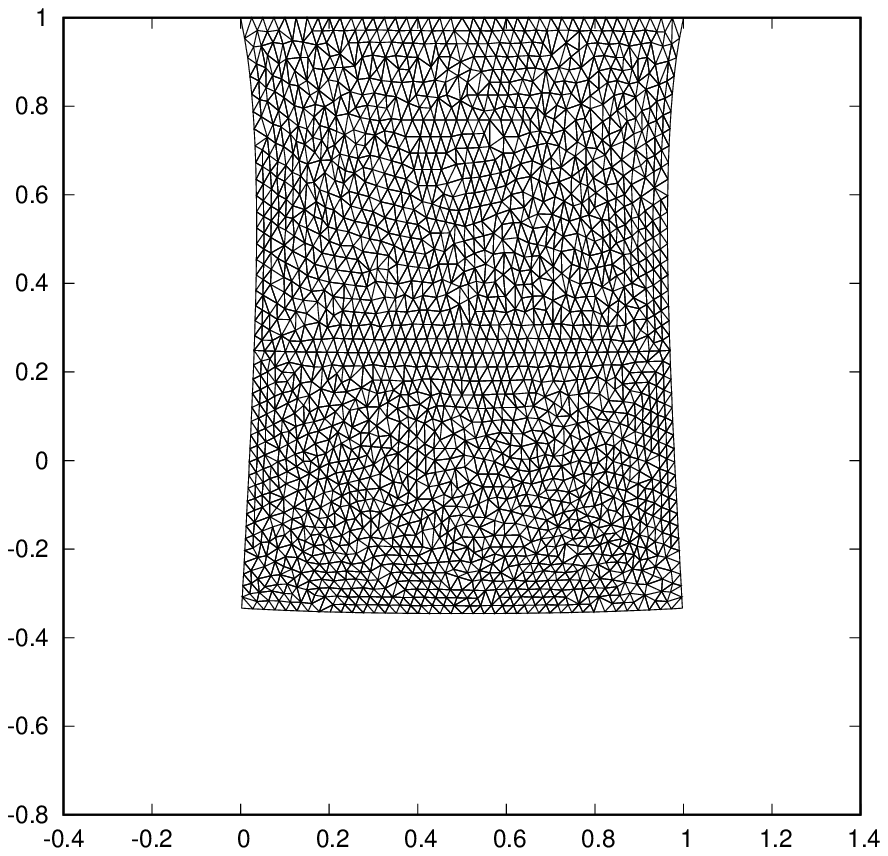}
\bigskip\\
(a4)~\includegraphics[bb=100 50 360 302,scale=.36]{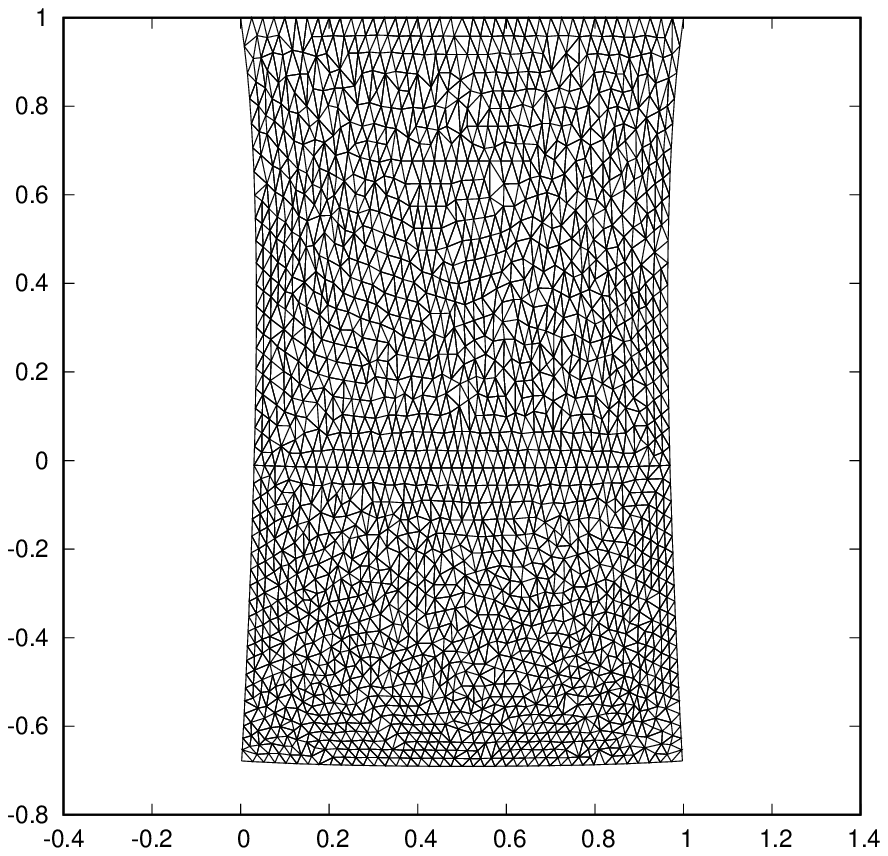}
(b4)~\includegraphics[bb=100 50 360 302,scale=.36]{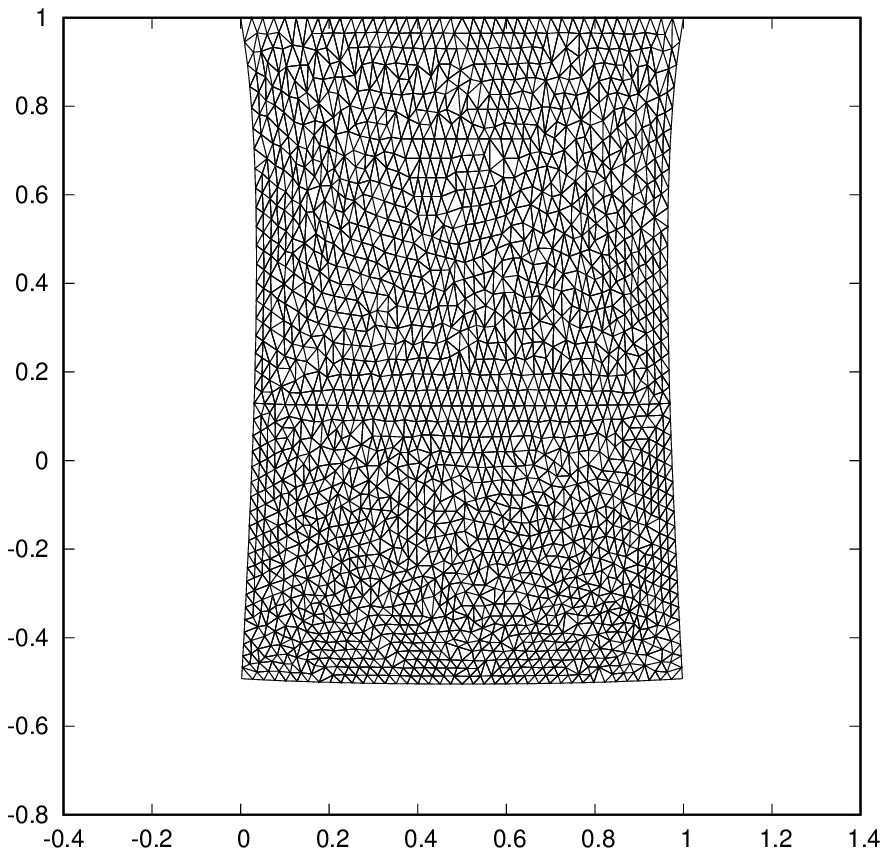}
(c4)~\includegraphics[bb=100 50 360 302,scale=.36]{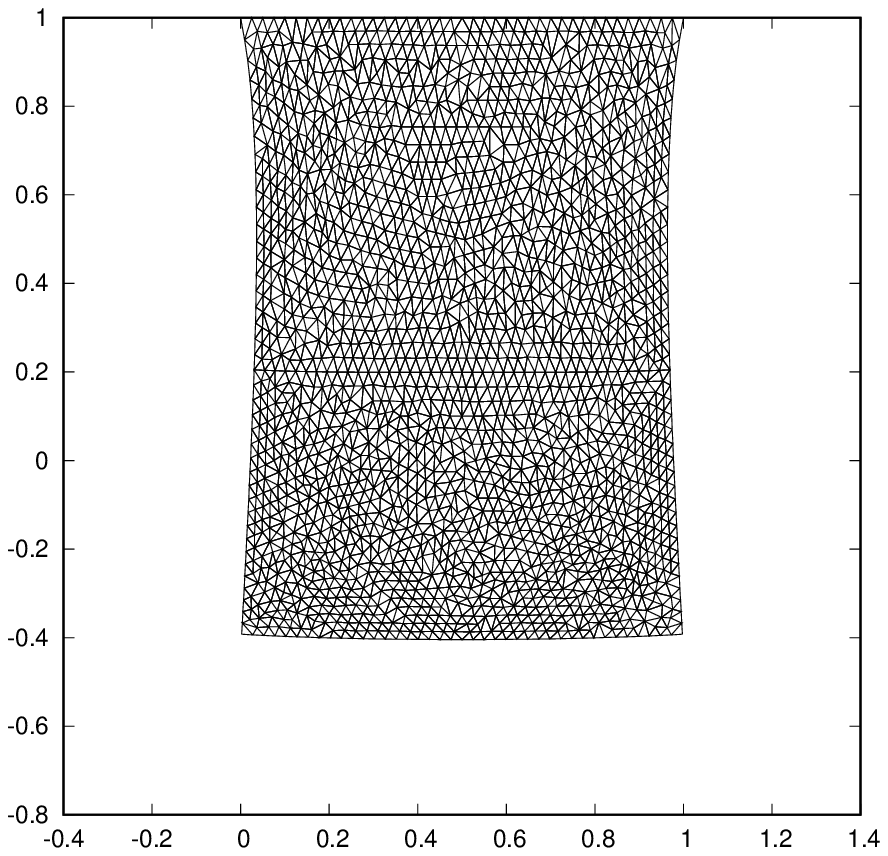}
\caption{Time evolution of the deformed shapes of the domain for $t=0.0$, $0.1$, $0.3$, $0.5$ and $1.0$ (top to bottom) for Example~\ref{ex:1}: left~(a): $\alpha=0$, center~(b): $\alpha=1$, right~(c): $\alpha=2$.}
\label{fig:time_evolution_ex1}
\end{figure}
\begin{figure}[!htbp]
\centering
\includegraphics[bb=50 50 410 302,scale=.5]{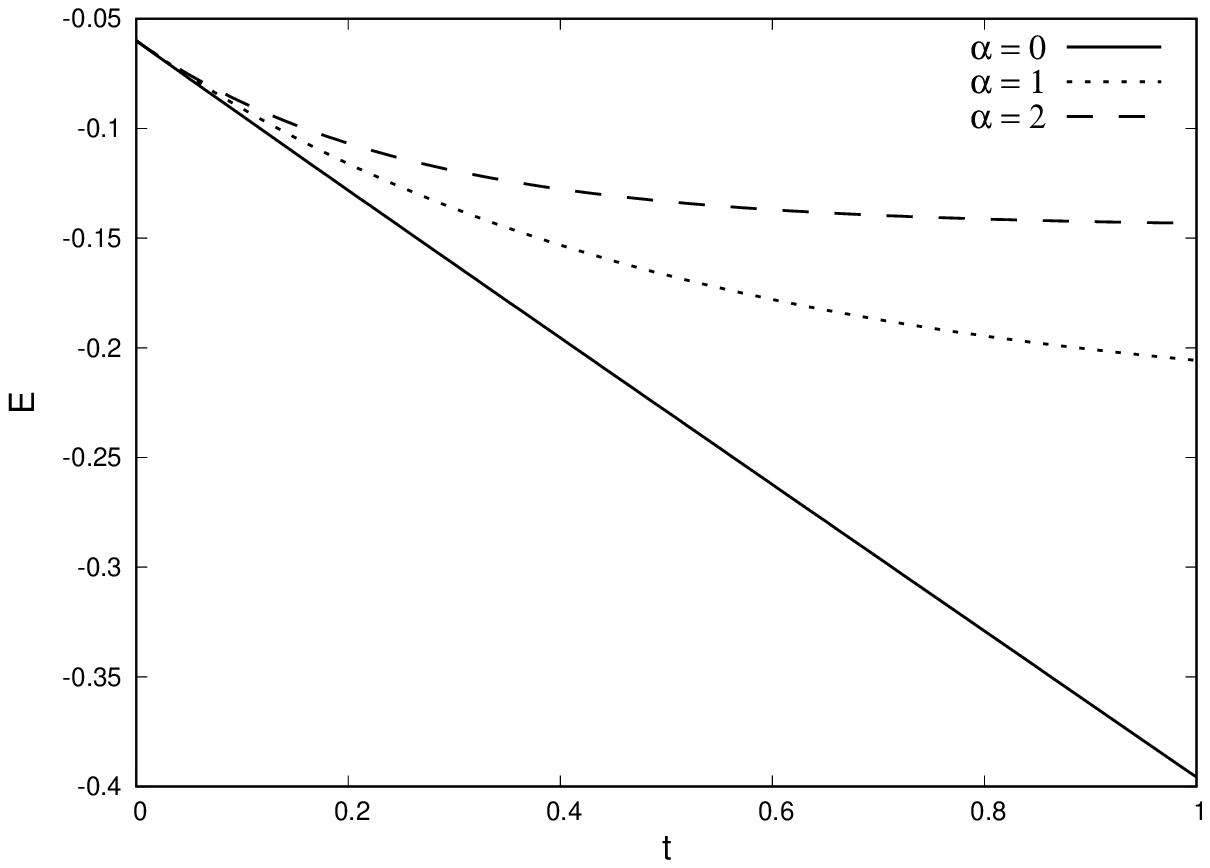}
\bigskip
\caption{The energy~$E_h^k~(k=0,\ldots, N_T)$ as a function of time for Example~\ref{ex:1}.}
\label{fig:energy_ex1}
\end{figure}
\begin{figure}[!htbp]
\centering
(a0)~\includegraphics[bb=100 50 360 302,scale=.36]{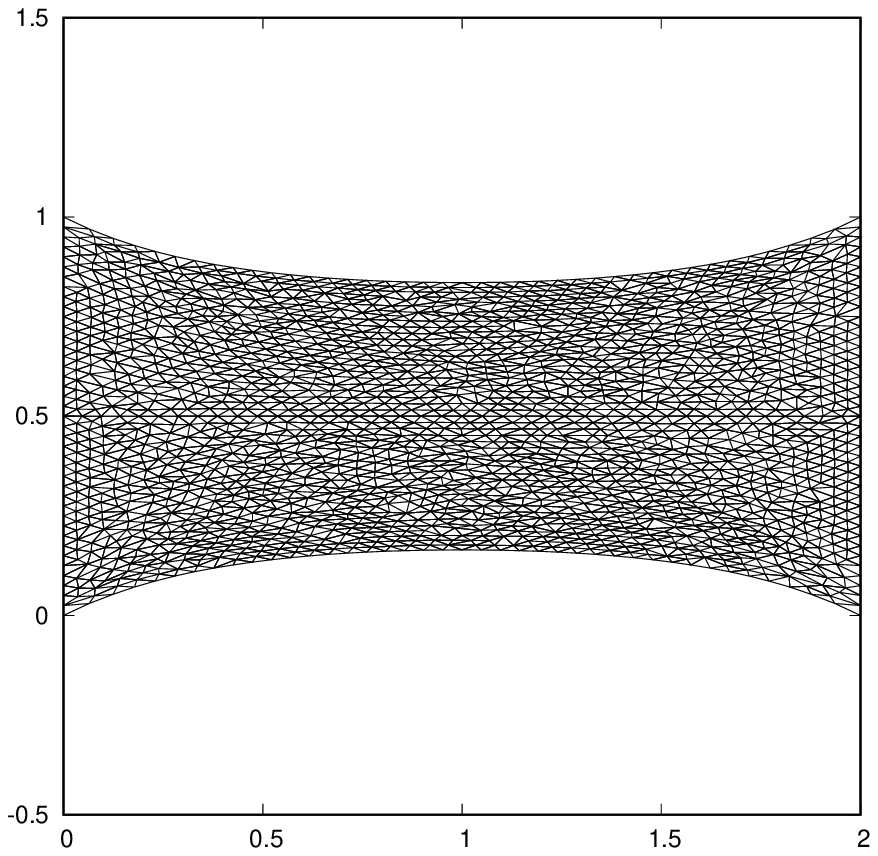}
(b0)~\includegraphics[bb=100 50 360 302,scale=.36]{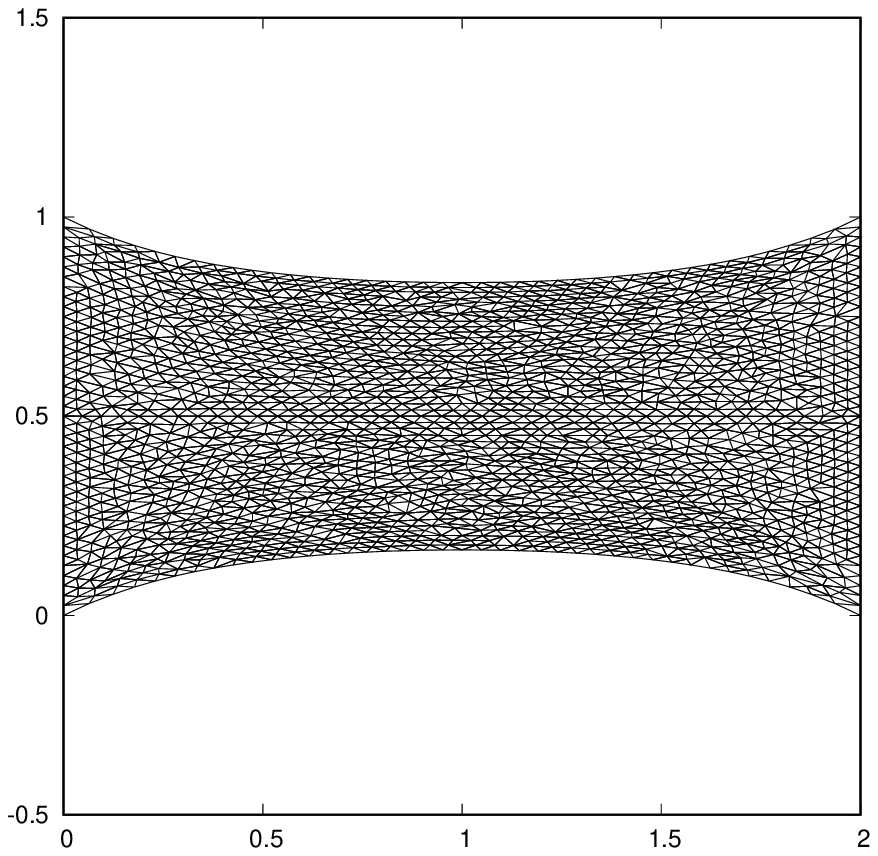}
(c0)~\includegraphics[bb=100 50 360 302,scale=.36]{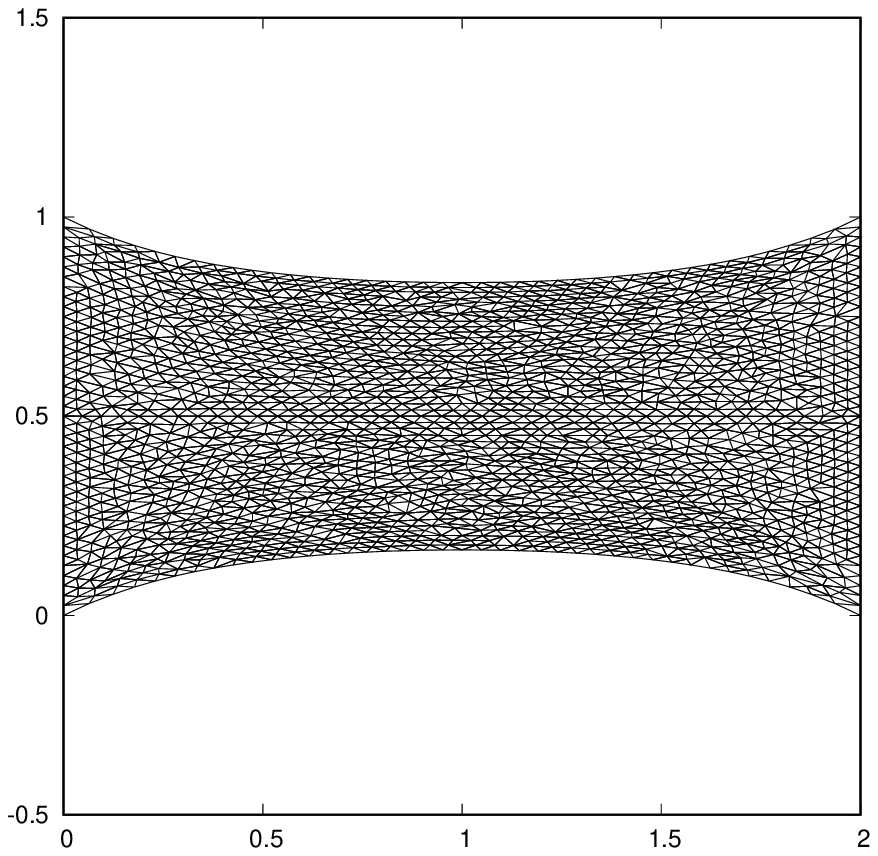}
\bigskip\\
(a1)~\includegraphics[bb=100 50 360 302,scale=.36]{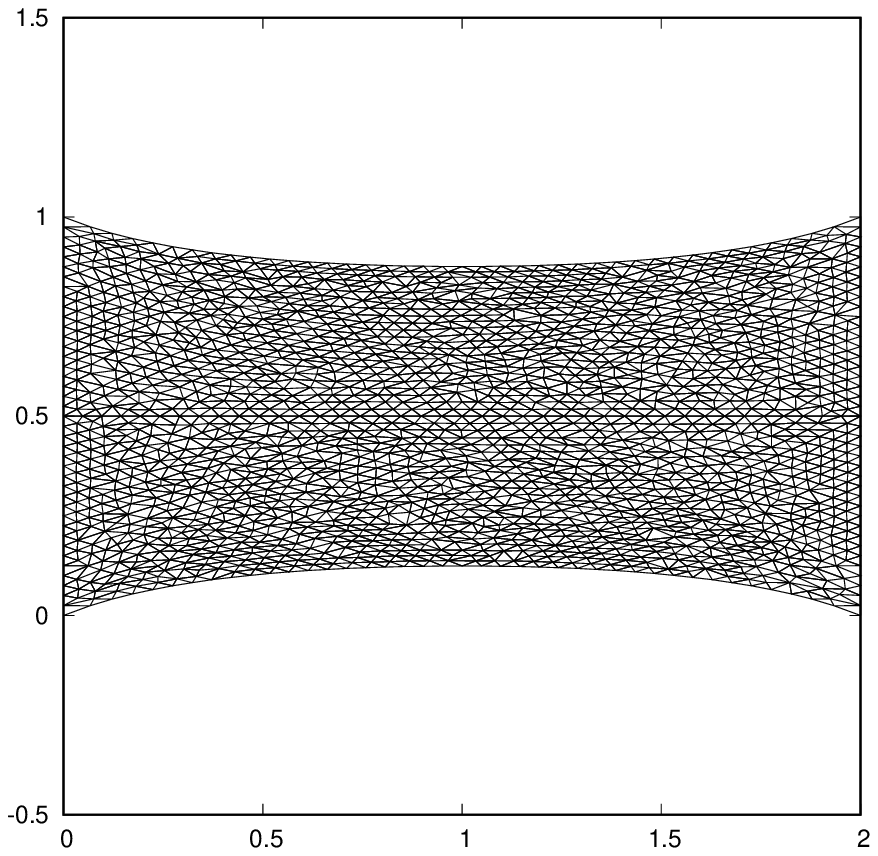}
(b1)~\includegraphics[bb=100 50 360 302,scale=.36]{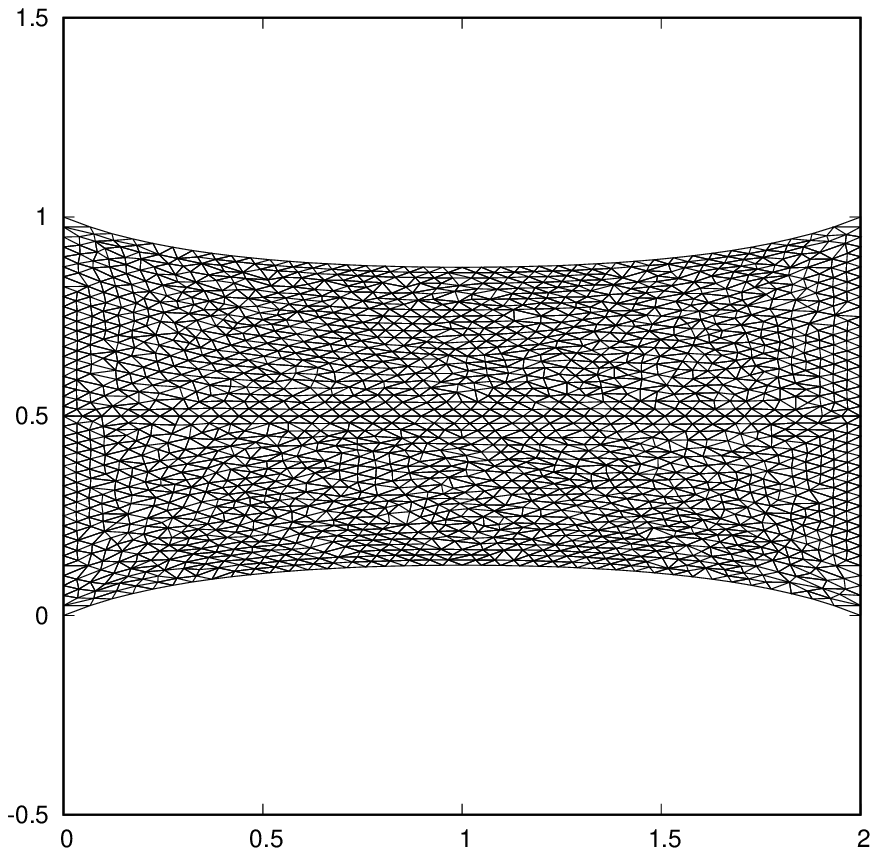}
(c1)~\includegraphics[bb=100 50 360 302,scale=.36]{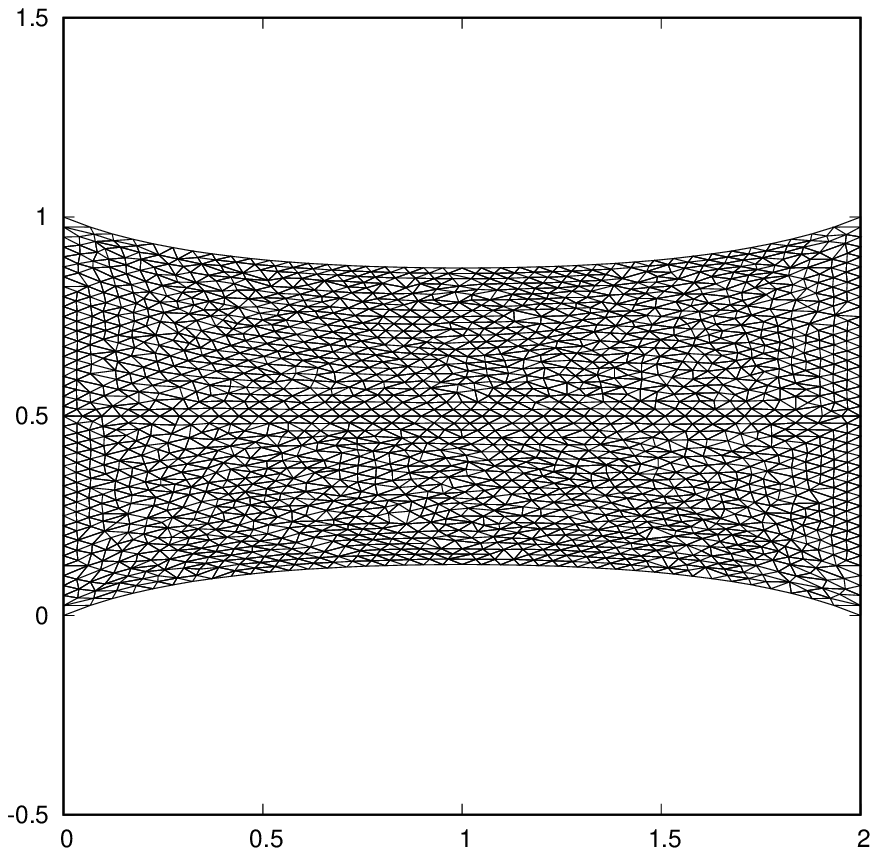}
\bigskip\\
(a2)~\includegraphics[bb=100 50 360 302,scale=.36]{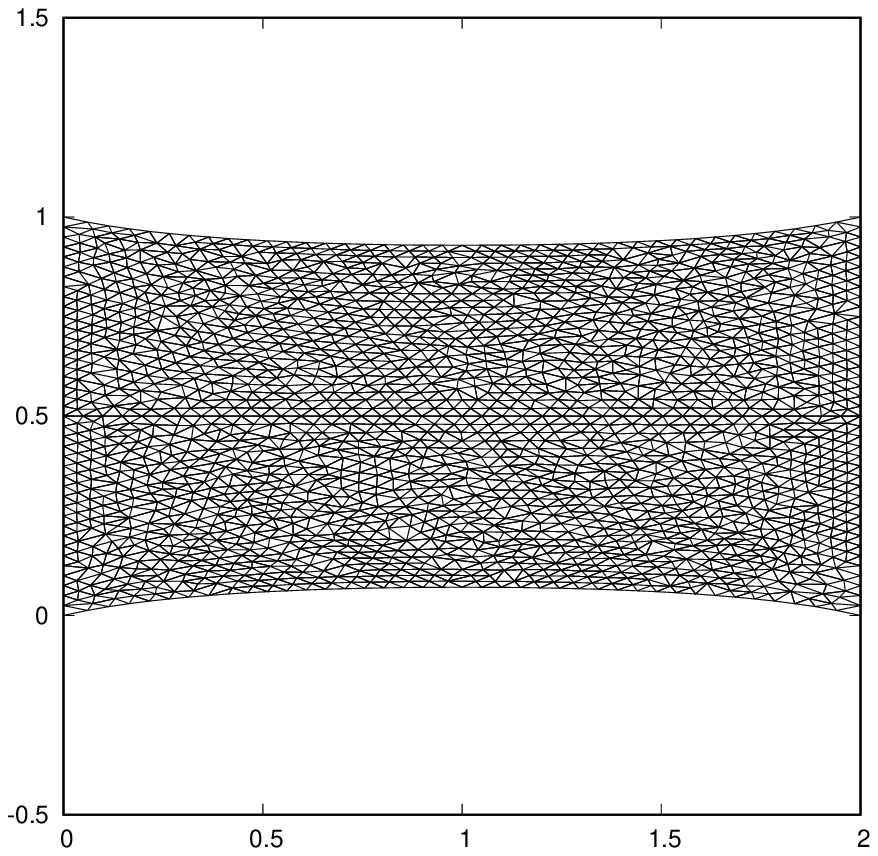}
(b2)~\includegraphics[bb=100 50 360 302,scale=.36]{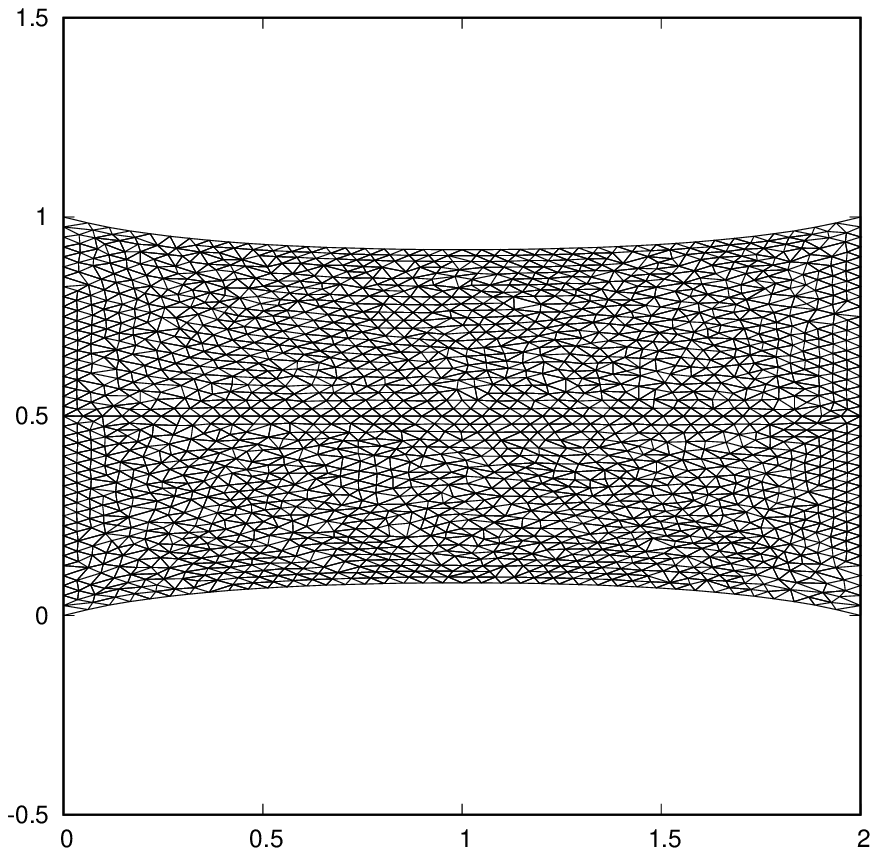}
(c2)~\includegraphics[bb=100 50 360 302,scale=.36]{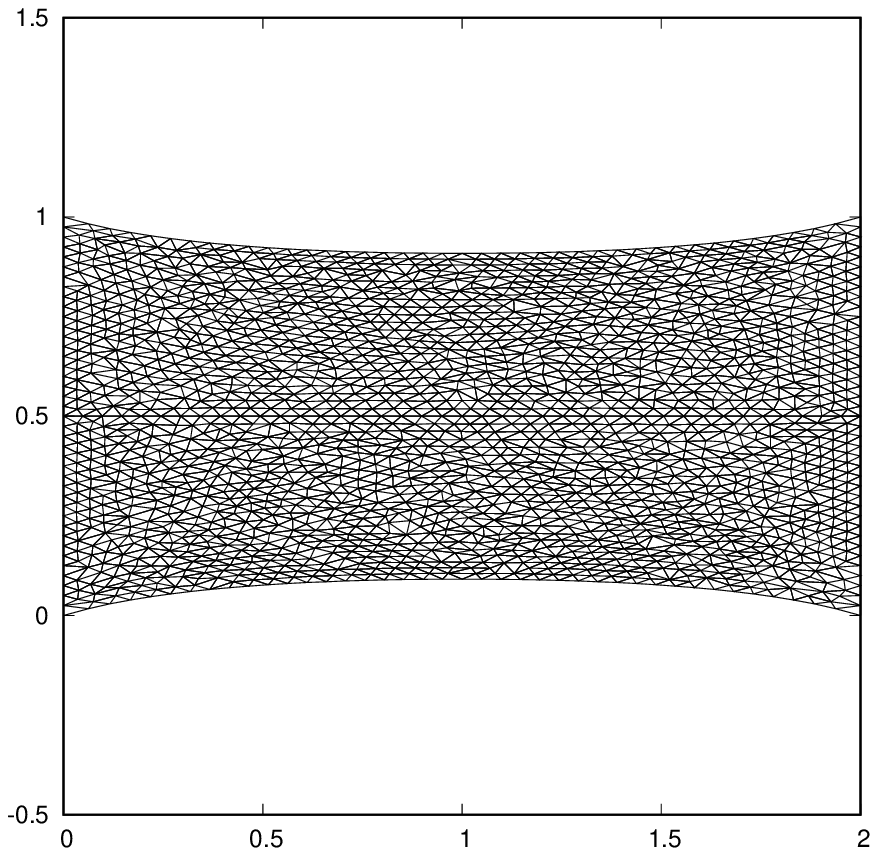}
\bigskip\\
(a3)~\includegraphics[bb=100 50 360 302,scale=.36]{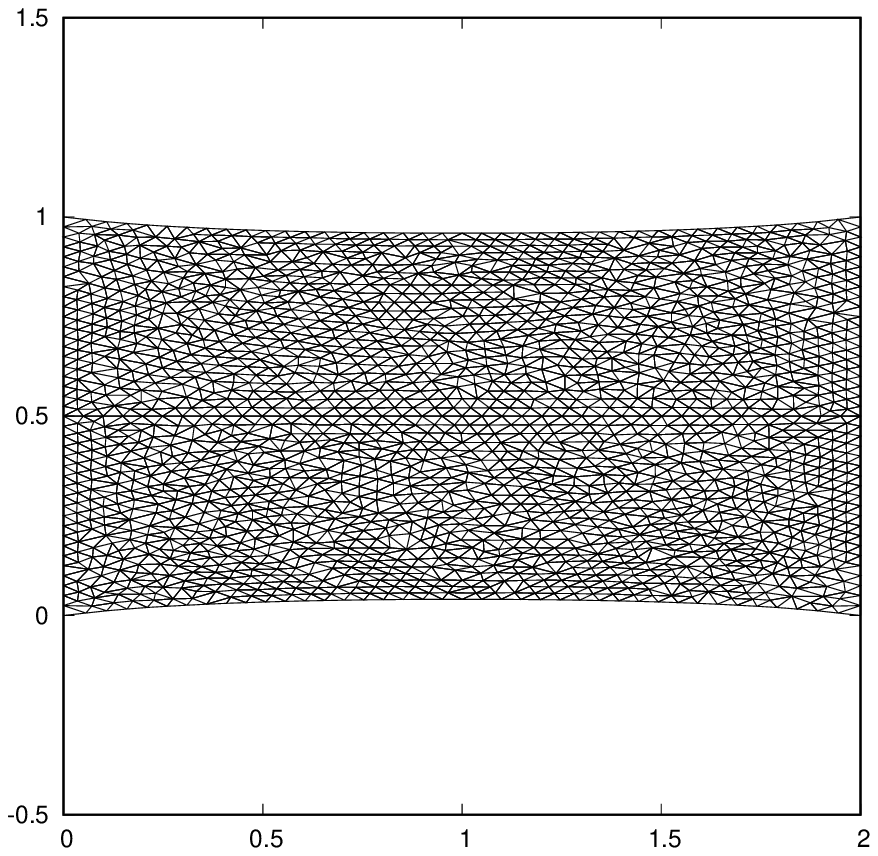}
(b3)~\includegraphics[bb=100 50 360 302,scale=.36]{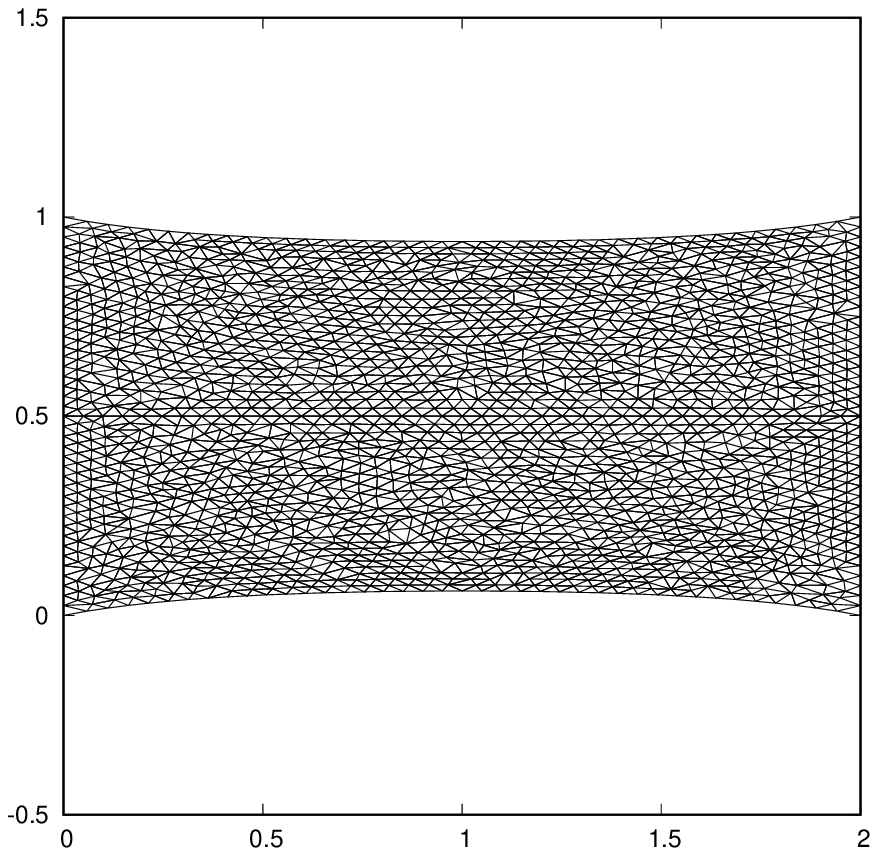}
(c3)~\includegraphics[bb=100 50 360 302,scale=.36]{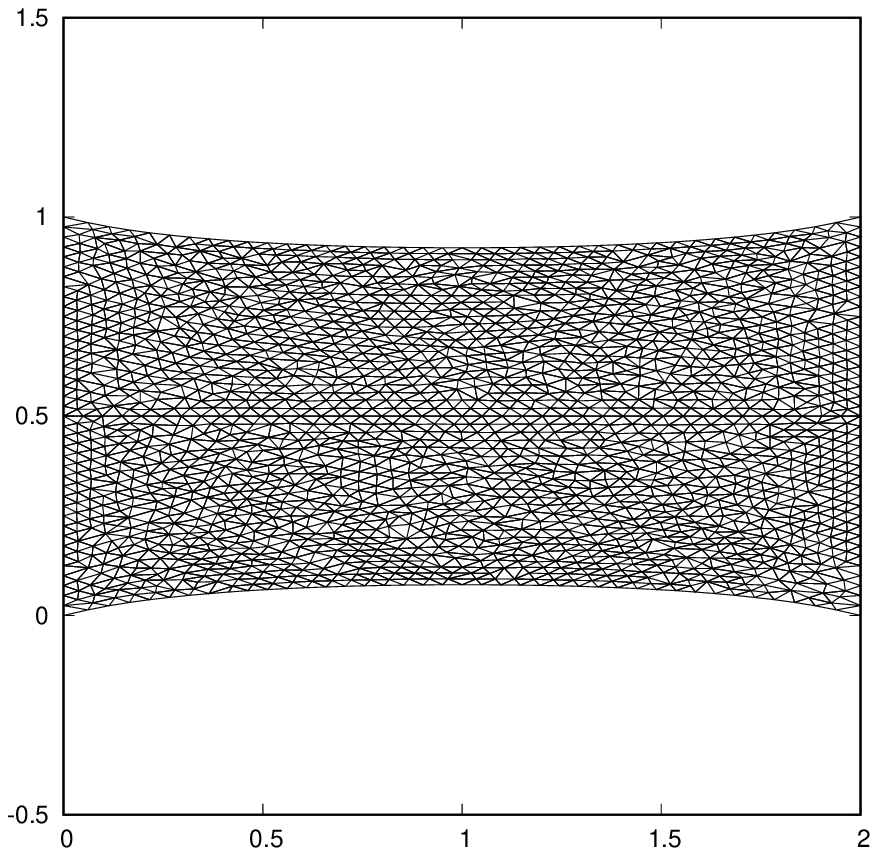}
\bigskip\\
(a4)~\includegraphics[bb=100 50 360 302,scale=.36]{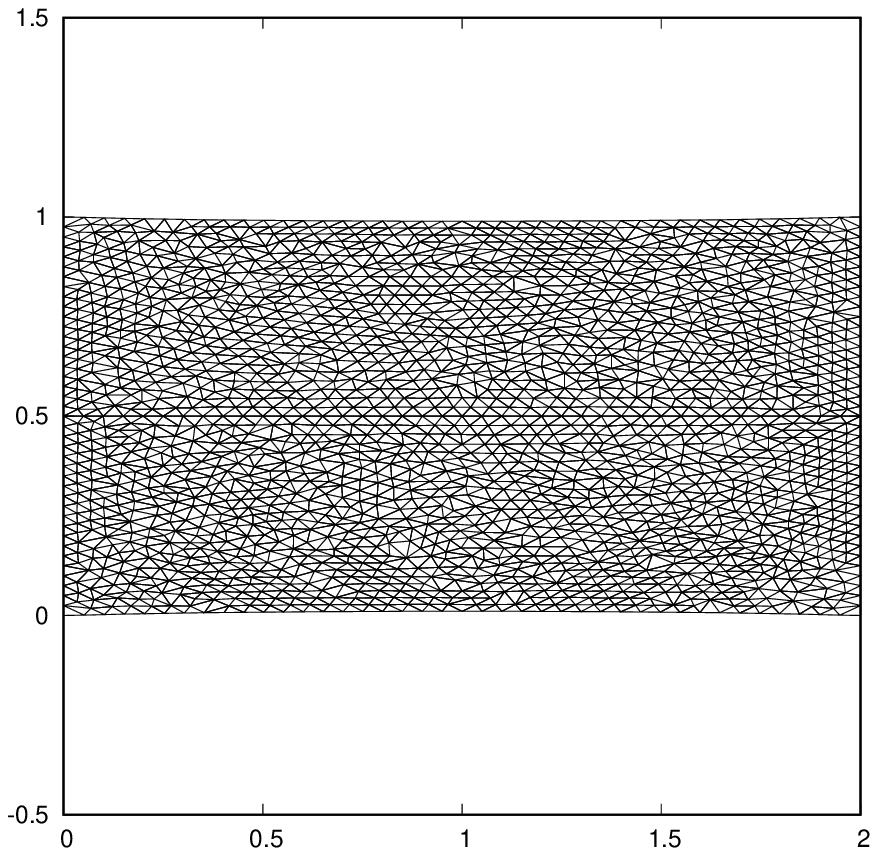}
(b4)~\includegraphics[bb=100 50 360 302,scale=.36]{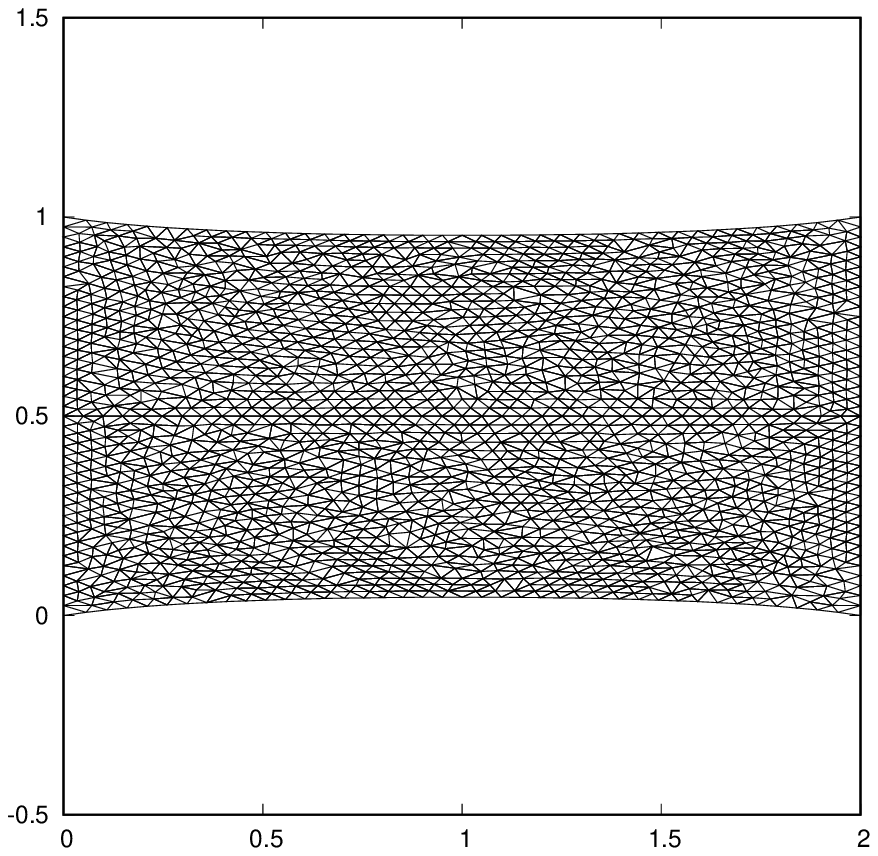}
(c4)~\includegraphics[bb=100 50 360 302,scale=.36]{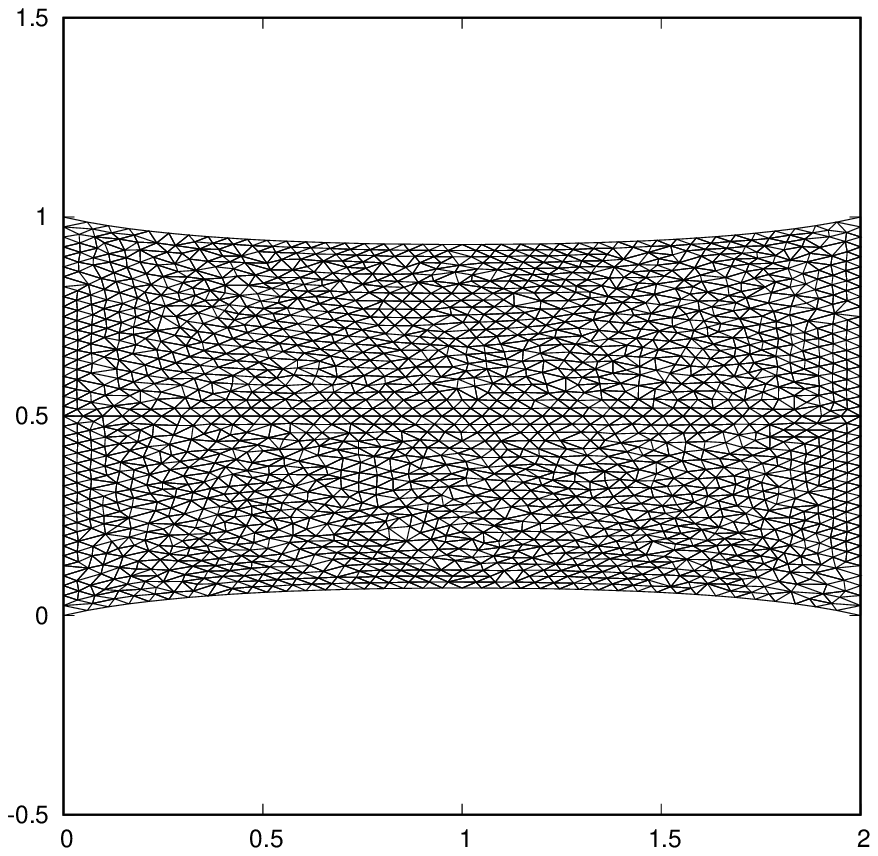}
\caption{Time evolution of the deformed shapes of the domain for $t=0.0$, $0.1$, $0.3$, $0.5$ and $1.0$ (top to bottom) for Example~\ref{ex:2}: left~(a): $\alpha=0$, center~(b): $\alpha=1$, right~(c): $\alpha=2$.}
\label{fig:time_evolution_ex2}
\end{figure}
\begin{figure}[!htbp]
\centering
\includegraphics[bb=50 50 410 302,scale=.5]{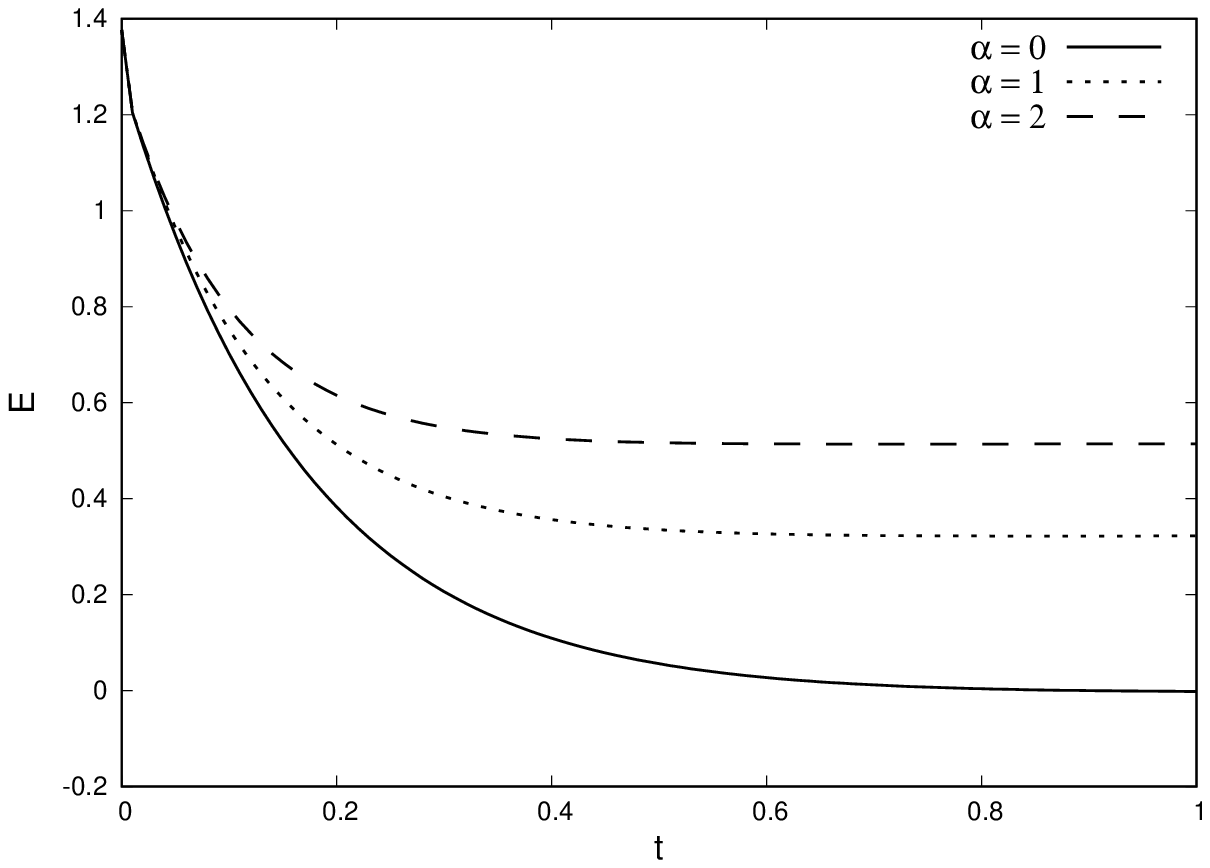}
\bigskip
\caption{The energy~$E_h^k~(k=0,\ldots, N_T)$ as a function of time for Example~\ref{ex:2}.}
\label{fig:energy_ex2}
\end{figure}
\begin{figure}[!htbp]
\centering
\includegraphics[bb=50 50 410 302,scale=.5]{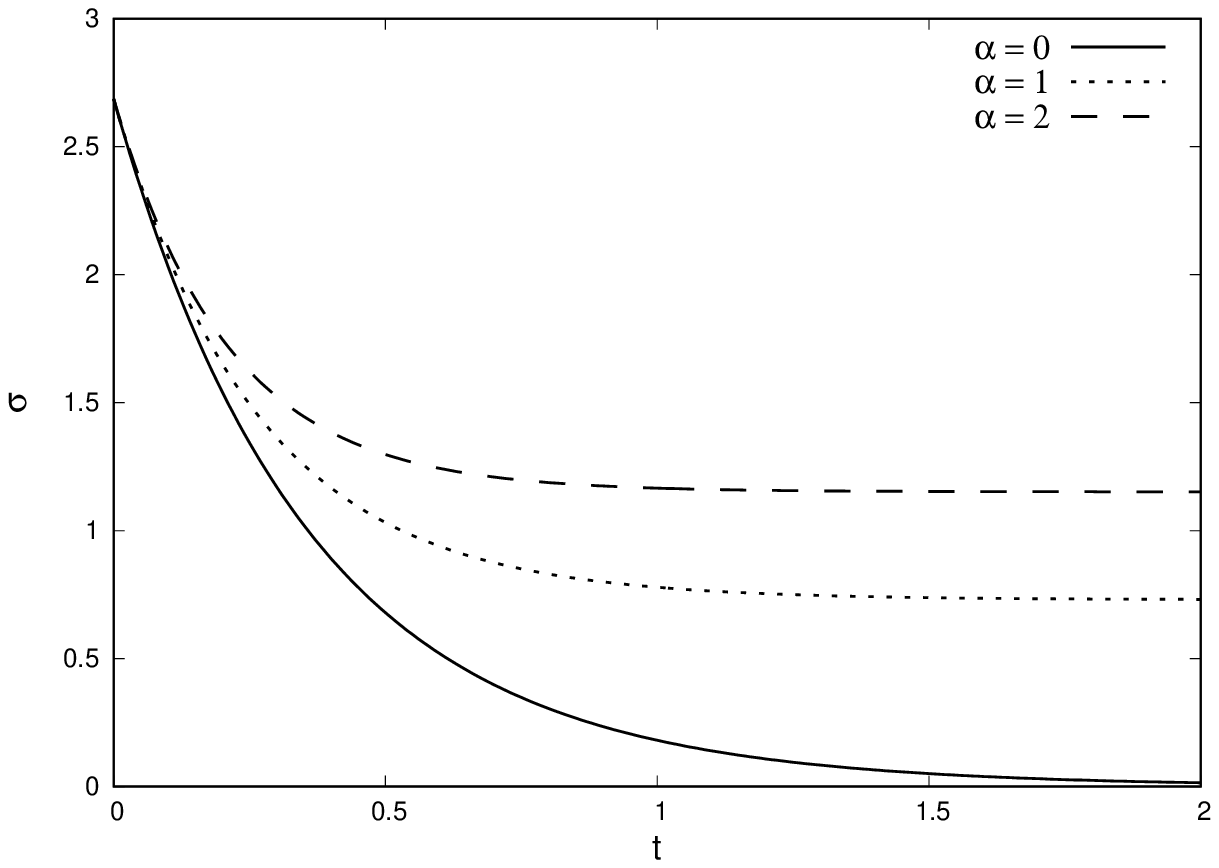}
\bigskip
\caption{$\|\sigma_{11} [u_h^k,\phi_h^k] \|_{L^\infty(\Omega)}~(k=0, \ldots, N_T)$ as a function of time for Example~\ref{ex:2}.}
\label{fig:stress_relax}
\end{figure}
%
%
%
%
%
%%%%%%%%%%%%%%%%%%%%%%%%%%%%%%%%%%%%%%%%%%%%%%%%%%%%%%%%%%%%
\section{Conclusions}\label{sec:conclusions}
%%%%%%%%%%%%%%%%%%%%%%%%%%%%%%%%%%%%%%%%%%%%%%%%%%%%%%%%%%%%
We have developed a gradient flow structure and established an energy decay property for the extended Maxwell viscoelastic model in Theorem~\ref{thm:gradient_flow_continuous}.
For a backward Euler time-discretization of the model, we have proved the existence and uniqueness of its solutions in Theorem~\ref{thm:solvability_tau} and established the time-discrete gradient flow structure of the corresponding energy in Theorem~\ref{thm:gradient_flow_tau}.
A P1/P0 finite element scheme preserving the structure has been presented, where the solvability and the stability in the sense of energy have been ensured in Theorems~\ref{thm:solvability_fem} and~\ref{thm:gradient_flow_fem}, respectively.
The backward Euler method has been employed for the time integration in the scheme.
The scheme is, however, realized by an efficient algorithm, cf.~Algorithm on p.\pageref{AL}, where for each time-step the function~$\phi_h^k$ is determined explicitly on each triangular element.
Two-dimensional numerical results have been shown to observe the typical viscoelastic phenomena, {\it creep} and {\it stress relaxation} and the effect of the relaxation parameter~$\alpha$.
\par
The existence and uniqueness of the Maxwell model~\eqref{prob:zener_strong} and the error estimates of the scheme will be presented in a forthcoming paper.
%
%
%
%%%%%%%%%%%%%%%%%%%%%%%%%%%%%%%%%%%%%%%%%%%%%
\section*{Acknowledgements}
%%%%%%%%%%%%%%%%%%%%%%%%%%%%%%%%%%%%%%%%%%%%%
%
This work is partially supported by JSPS KAKENHI Grant Numbers JP16H02155, JP17H02857, JP26800091, JP16K13779, JP18H01135, and JP17K05609, JSPS A3 Foresight Program, and JST PRESTO Grant Number~JPMJPR16EA.
%
%
%
%
%%%%%%%%%%%%%%%%%%%%%%%%%%%%%%%%%%%%%%%%%%%%%%%%%%%%%%%%%%%%
% References
%%%%%%%%%%%%%%%%%%%%%%%%%%%%%%%%%%%%%%%%%%%%%%%%%%%%%%%%%%%%
\newcommand{\noop}[1]{}


\begin{thebibliography}{10}

\bibitem{AbuEbe-2007}
O.M. Abuzeid and P.~Eberhard.
\newblock Linear viscoelastic creep model for the contact of nominal flat
  surfaces based on fractal geometry: standard linear solid {(SLS)} material.
\newblock {\em Journal of Tribology}, 129:461--466, 2007.

\bibitem{Cia-1978}
P.G. Ciarlet.
\newblock {\em {T}he {F}inite {E}lement {M}ethod for {E}lliptic {P}roblems}.
\newblock North-Holland, Amsterdam, 1978.

\bibitem{Ferry-1970}
J.D. Ferry.
\newblock {\em {V}iscoelastic {P}roperties of {P}olymers}.
\newblock Wiley, New York, 1970.

\bibitem{GolGra-1988}
J.M. Golden and G.A.C. Graham.
\newblock {\em {B}oundary {V}alue {P}roblems in {L}inear {V}iscoelasticity}.
\newblock Springer, Berlin, 1988.

\bibitem{FreeFem}
F.~Hecht.
\newblock New development in {F}ree{F}em++.
\newblock {\em Journal of Numerical Mathematics}, 20(3-4):251--265, 2012.

\bibitem{KarShaWarWhi-2005}
M.~Karamanou, S.~Shaw, M.K. Warby, and J.R. Whiteman.
\newblock Models, algorithms and error estimation for computational
  viscoelasticity.
\newblock {\em Computer Methods in Applied Mechanics and Engineering},
  194(2-5):245--265, 2005.

\bibitem{KNTY_prep}
M.~Kimura, H.~Notsu, Y.~Tanaka, and H.~Yamamoto.
\newblock In preparation.

\bibitem{Lockett-1972}
F.J. Lockett.
\newblock {\em {N}onlinear {V}iscoelastic {S}olids}.
\newblock Academic Press, Paris, 1972.

\bibitem{Mac-1994}
C.W. Macosko.
\newblock {\em Rheology: Principles, Measurements, and Applications}.
\newblock Wiley-VCH, New York, 1994.

\bibitem{Nec-1967}
J.~Ne\v{c}as.
\newblock {\em {L}es {M}\'ethods {D}irectes en {T}h\'eories des {\'E}quations
  {E}lliptiques}.
\newblock Masson, Paris, 1967.

\bibitem{RivSha-2006}
B.~Rivi{\`e}re and S.~Shaw.
\newblock Discontinuous {G}alerkin finite element approximation of nonlinear
  non-{F}ickian diffusion in viscoelastic polymers.
\newblock {\em SIAM Journal on Numerical Analysis}, 44(6):2650--2670, 2006.

\bibitem{RivShaWheWhi-2003}
B.~Rivi{\`e}re, S.~Shaw, M.F. Wheeler, and J.R. Whiteman.
\newblock Discontinuous galerkin finite element methods for linear elasticity
  and quasistatic linear viscoelasticity.
\newblock {\em Numerische Mathematik}, 95(2):347--376, 2003.

\bibitem{RogWin-2010}
M.E. Rognes and R.~Winther.
\newblock Mixed finite element methods for linear viscoelasticity using weak
  symmetry.
\newblock {\em Mathematical Models and Methods in Applied Sciences},
  20:955--985, 2010.

\bibitem{ALBERTA-2005}
A.~Schmidt and K.G. Siebert.
\newblock {\em {D}esign of {A}daptive {F}inite {E}lement {S}oftware: {T}he
  {F}inite {E}lement {T}oolbox {ALBERTA}}.
\newblock Springer, Berlin, 2005.

\bibitem{ShaWhi-2004}
S.~Shaw and J.R. Whiteman.
\newblock A posteriori error estimates for space-time finite element
  approximation of quasistatic hereditary linear viscoelasticity problems.
\newblock {\em Computer Methods in Applied Mechanics and Engineering},
  193(52):5551--5572, 2004.

\end{thebibliography}
\end{document}